\documentclass[de, en, amspaper, secnum, nobiblatex]{./cls/myamspaper}

\usepackage[top=40mm, bottom=40mm, left=30mm, right=30mm]{geometry}


\newcommand{\pt}{\mathrm{pt}}
\newcommand{\Gr}{\operatorname{Gr}}
\newcommand{\sign}{\operatorname{sign}}

\DeclareMathAlphabet{\mathpzc}{OT1}{pzc}{m}{it}
\newcommand{\pzg}{\mathpzc{g}}
\newcommand{\pzh}{\mathpzc{h}}
\newcommand{\pzu}{\mathpzc{u}}
\newcommand{\pzv}{\mathpzc{v}}

\DeclarePairedDelimiter\floor{\lfloor}{\rfloor}

\newcommand{\euE}{\EuScript{E}}
\newcommand{\euF}{\EuScript{F}}
\newcommand{\euG}{\EuScript{G}}
\newcommand{\euH}{\EuScript{H}}

\newcommand{\euS}{\EuScript{S}}
\newcommand{\euT}{\EuScript{T}}

\usepackage{marginnote}

\title[Uniform enveloping semigroupoids for groupoid actions]{Uniform enveloping semigroupoids for groupoid actions}

\begin{document}

\begin{abstract}
  We establish new characterizations for (pseudo)isometric extensions
  of topological dynamical systems. For such extensions, we also extend results about 
  relatively invariant measures and Fourier analysis that were previously 
  only known in the minimal case to a significantly larger 
  class, including all transitive systems. To bypass the reliance on
  minimality of the 
  classical approaches to isometric extensions via the 
  Ellis semigroup, we show that extensions of 
  topological dynamical systems can be described as \emph{groupoid actions} and then 
  adapt the concept of enveloping semigroups to construct a 
  \emph{uniform enveloping semigroupoid} for groupoid actions. This approach 
  allows to deal with the more complex orbit structures of nonminimal systems.
  
  We study uniform enveloping semigroupoids of general groupoid
  actions and translate the results back to the special case of extensions
  of dynamical systems. In particular, we show that, under appropriate 
  assumptions, a groupoid action is (pseudo)isometric if and only if 
  the uniform enveloping semigroupoid is actually a compact groupoid. 
  We also provide an operator theoretic characterization based on 
  an abstract Peter--Weyl-type theorem for representations of compact, transitive 
  groupoids on Banach bundles which is of independent interest.
\end{abstract}

\maketitle

Given a topological dynamical system $(K, \phi)$ consisting of a compact 
space $K$ and a homeomorphism $\phi\colon K\to K$, its \emph{enveloping 
Ellis semigroup} $\uE(K, \phi)$ introduced by R.\ Ellis in \cite{Ellis1960} is the 
pointwise closure
\begin{align*}
  \uE(K, \phi) \defeq \overline{\left\{ \phi^n \mmid n\in\Z \right\}} \subset K^K.
\end{align*}
It is an important tool in topological dynamics capturing the 
long-term behavior of a dynamical system. Moreover, it allows to study 
the system $(K, \phi)$ via algebraic properties of $\uE(K, \phi)$. In particular
the elegant theory of compact, right-topological semigroups has been used 
to describe and study properties of topological dynamical systems. 
We refer to \cite[Chapters 3 and 6]{Ausl1988} and \cite{Glasner2007}
for the general theory of the Ellis semigroup and to 
\cite[Chapter 5]{AperiodicOrder2015}, \cite{Staynova2019}, or 
\cite[Section 4]{GlasnerGutmanYe2018} for some recent applications.

In the special case of an equicontinuous system, $\mathrm{E}(K,\varphi)$ is actually a compact topological group which agrees with the \emph{uniform enveloping semigroup}
\begin{align*}
  \uE_\uu(K, \phi) \defeq \overline{\left\{ \phi^n \mmid n\in\Z \right\}} \subset \mathrm{C}(K,K)
\end{align*}
where the closure is taken with respect to the the compact-open topology on $\mathrm{C}(K,K)$, i.e., the topology of uniform convergence. In this case 
one can use the Peter--Weyl theorem to understand the representation 
of the compact group $\uE(K, \phi)$ on $\uC(K)$.
In particular, one can prove the following characterizations of equicontinuous 
systems $(K,\varphi)$ involving the \emph{Koopman operator} 
$T_\phi\colon \uC(K)\to\uC(K), \, f \mapsto f \circ \phi$.

\begin{theorem*}\label{thm:charequi}
  For a topological dynamical system $(K,\varphi)$ the following assertions are equiavalent.
    \begin{enumerate}[(a)]
      \item $(K,\varphi)$ is equicontinuous.
      \item $\uE_\uu(K, \phi)$ is a compact group.
      \item The Koopman operator $T_\varphi$ has 
      \emph{discrete spectrum}, i.e., the union of its 
      eigenspaces is total in $\mathrm{C}(K)$.
    \end{enumerate}
\end{theorem*}


The main goals of this article is to develop the techniques to prove 
an anlogous statement for \enquote{structured} extensions 
\begin{align*}
  q\colon (K, \phi) \to (L, \psi)
\end{align*}
of dynamical systems. The importance of these extensions
is in particular due to the famous structure theorem for distal minimal flows proved by H.\ Furstenberg. It states that any distal minimal flow can be constructed via a 
\emph{Furstenberg tower} consisting of equicontinuous 
(equivalently: pseudoisometric) extensions 
(see, e.g., \cite[Section V.3]{deVr1993}). Beyond this result, such 
extensions (especially the case of compact group extensions) 
have continued to play an important role in the structure theory of 
dynamical systems (see \cite{HoKr2018}, \cite{Ziegler2007}), the construction
of new dynamical systems (see, e.g., \cite{Dolgopyat2002} or \cite[Section 6]{GHWS2020}), and applications
to number theory (see \cite{FKLM2016}).

First steps towards a characterization of pseudoisometric extensions were made by 
A.\ W.\ Knapp in \cite{Knapp1967}, though his results use the Ellis semigroup
and only cover minimal distal systems, making essential use of minimality.
We suggest to work around this built-in dependence on recurrence and propose a 
new, more general approach to structured extensions 
$q \colon (K,\varphi) \to (L,\psi)$. Instead of looking at an extension as 
a morphism 
between two \emph{group actions}, we show that an extension can 
be equivalently regarded as a single system defined by a \emph{groupoid action}. 
In analogy to enveloping semigroups, we introduce and study an 
\emph{enveloping semigroupoid} $\euE_\uu(q)$ to describe the structuredness of $q$.
This leads in particular to the following characterization of pseudoisometric
extensions.

\begin{theorem*}
  For an open extension $q \colon (K,\varphi) \to (L,\psi)$ of 
  topological dynamical systems such that $\dim\fix(T_\psi) = 1$, the 
  following assertions are equivalent.
    \begin{enumerate}[(a)]
      \item $q$ is pseudoisometric.
      \item The uniform enveloping semigroupoid $\euE_\uu(q)$ is a compact groupoid.
      \item The union of all finitely generated, projective, closed $T_\varphi$-invariant 
      $\uC(L)$-submodules of $\mathrm{C}(K)$ is dense in $\uC(K)$.
    \end{enumerate}
\end{theorem*}

Note that previously merely the implication 
(a) $\implies$ (c) was somewhat known and only
in the minimal distal case (see \cite[Theorem 1.2]{Knapp1967}). The condition that $T_\psi$ has no 
nonconstant fixed functions covers a significantly larger family than minimal
systems, including all transitive systems but also many dynamical systems
with more complex orbit structures.

Groupoids are generalizations of groups that allow to capture local symmetry. They play an 
important role in noncommutative geometry by
providing a framework for studying operator algebras, index theory, and foliations (see 
\cite{Connes1995} or \cite{MooreSchochet2006}).
In ergodic theory, 
G.\ W.\ Mackey used groupoids for his theory of virtual groups in order \enquote{to bring to light and 
exploit certain apparently far reaching analogies between group theory and ergodic theory} 
(\cite[p.\ 187 and Section 11]{Mackey1966}). And in topological 
dynamics, they have long been used as a bridge between dynamics and $\uC^*$-algebras
in order to study questions around orbit equivalence, see 
\cite{Tomiyama1987}, \cite{GiordanoPutnamSkau1995}, 
and more recently \cite{MatsumotoMatui2014}.

It is the goal of this article to demonstrate that groupoids also provide
a natural approach to extensions of topological dynamical systems and 
that the systematic analysis of the occurring groupoid structures
allows to simplify and generalize existing results on (pseudo)isometric and 
equicontinuous extensions. In the process, we investigate the representation
theory of compact, transitive groupoids and prove theorems in 
\cref{sec:rep} that 
are of independent interest, including a general Peter-Weyl type theorem in 
\cref{pwtransitive}.
Beyond the above-mentioned characterization, we also apply our abstract 
results to prove the existence of relatively invariant measures for certain
pseudoisometric extensions (see \cref{ext3}) and to show that, much as in 
the case of equicontinuous systems, pseudoisometric extensions admit 
Fourier-analytic decompositions, see \cref{ext4}.

\textbf{Organization of the article.} Since all results on extensions of dynamical systems proved in this article depend 
only on groupoid structure, most results are formulated in the more general framework of groupoids 
and their actions. For the reader's convenience, however, the main results (the
applications to extensions of topological dynamical systems) are gathered 
in \cref{sec:applications}.

In \cref{sec:groupoids}, we recall the concepts of (semi)groupoids 
and their actions and show in \cref{ex:grpaction} and \cref{ex:extensiongrpdaction} how an extension of dynamical systems can be 
described as a groupoid action. We then begin generalizing concepts for extensions
to the context of general groupoid actions (see \cref{def:struct_ext}).
\cref{sec:topology} is devoted to a generalization of the 
compact-open topology in \cref{def:comopen} which we will need to define the 
uniform enveloping semigroupoid. In particular, we prove a 
charaterization of convergence for nets of mappings defined on distinct 
fibers of a bundle (see \cref{charconv}). This plays a key role 
throughout the article. After these preparations, \cref{sec:enveloping_semigroupoids}
then introduces the 
\emph{uniform enveloping semigroupoid} of a groupoid action as a 
generalization of the uniform enveloping semigroup for group actions, 
see \cref{def:uniformenvsgrpd}. 
We use the generalized Arzel\`a-Ascoli theorem \cref{arzelaascoli} to show 
in \cref{thm:pseudoisochar}
that---under an assumption of topological ergodicity---a groupoid action is 
pseudoisometric if and only if its uniform enveloping 
semigroupoid is a compact groupoid. 

To explain what this compactness means on the operator-theoretic level,
\cref{sec:rep} collects results about representations of compact groupoids 
on Banach bundles. We first prove a Peter--Weyl-type theorem for 
representations of compact transitive groupoids on Banach bundles in \cref{pwtransitive}. 
This is then applied to the uniform enveloping 
(semi)groupoid of pseudoisometric groupoid actions to derive the desired 
operator-theoretic 
characterizations of structuredness, one of our main results \cref{mainthm}.

In preparation of \cref{sec:applications}, \cref{sec:RIM} investigates the existence and 
uniqueness of relatively invariant measures for certain (pseudo)isometric groupoid actions, see \cref{RIM}.
In \cref{sec:applications}, we then prove Fourier analytic results for transitive 
actions of compact groupoids, generalizing the Fourier analysis of compact groups 
and their actions (see \cref{thm:fouriermain}). Via the uniform 
enveloping (semi)groupoids this can be used to obtain a better understanding of 
pseudosisometric groupoid actions. Finally, \cref{sec:applications} restates
all our main results in the case of extensions of 
topological dynamical systems.


\textbf{Terminology and Notation.} All compact spaces are assumed to be Hausdorff though we may 
occasionally specify the Hausdorff property for emphasis. The neighborhood filter of a point
$x\in X$ in a topological space $X$ is denoted by $\mathcal{U}_X(x)$ or simply $\mathcal{U}(x)$ when 
there is no room for ambiguity. If $X$ is a uniform space, we write $\mathcal{U}_X$ for 
the uniform structure of $X$. 

At several points in the paper we consider \emph{bundles}, i.e., 
continuous surjections 
$p \colon E \to L$ for some topological \emph{total space} $E$ (usually with some 
additional structure) to a topological (usually compact) \emph{base space} $L$. For $l\in L$, 
we write 
$E_l \defeq p^{-1}(l)$ for the \emph{fiber over $l$} of such a bundle. Moreover, if 
$p_1\colon E_1 \to L$ and $p_2 \colon E_2 \to L$ are two bundles over the same 
base space $L$, we define
\begin{align*}
  E_1 \times_{p_1,p_2} E_2 \defeq \{(x,y) \in E_1\times E_2 \mid p_1(x) = p_2(y)\} \subset E_1 \times E_2
\end{align*}
and equip this set with the subspace topology induced by the product topology on $E_1 \times E_2$. 
We also write $E_1 \times_L E_2 \defeq E_1 \times_{p_1,p_2} E_2$ if the mappings $p_1$ and $p_2$ are clear.

We use the letters $S$ and $G$ for semigroups and 
groups and the letters $\euS$ and $\euG$ for semigroupoids and groupoids, respectively.
By a \emph{topological dynamical system} we mean a triple $(K, G, \phi)$ consisting of a non-empty compact
space $K$, a Hausdorff topological group $G$, and a continuous action $\phi\colon G\times K\to K$ of $G$ on $K$.
For $g \in G$, we denote the map $\phi(g,\cdot)\colon K\to K$ by $\phi_g$.
We omit $\phi$ from the notation if there is no room for confusion and if $G = \Z$, we
abbreviate $(K,G,\phi)$ by $(K,\phi)$ and identify $\phi$ with the map $\phi(1, \cdot)\colon K\to K$ 
that completely determines the action.

As usual, a \emph{morphism} $q \colon (K,G,\phi) \to (L,G,\psi)$ between dynamical systems 
$(K,G,\phi)$ and $(L,G,\psi)$ is a continuous mapping $q\colon K \to L$ such that the diagram
\begin{align*}
    \xymatrix{
      K   \ar[r]^{\phi_g} \ar[d]_{q}   & K \ar[d]^{q} \\
      L   \ar[r]_{\psi_g}                 & L
    }
\end{align*}
commutes for all $g \in G$. A surjective morphism $q \colon (K,G,\phi) \to (L,G,\psi)$ is an 
\emph{extension (of topological dynamical systems)}.

Finally, if $K$ is a compact space, we write $\uC(K)$ for the Banach space of all continuous 
complex-valued functions on $K$. We identify its dual space $\uC(K)$ with the space of all 
complex regular Borel measures on $K$ and write $\uP(K) \subset \uC(K)'$ for the 
space of all probability measures in $\uC(K)'$. The Dirac measure defined by a point $x \in K$ is denoted by $\delta_x$. If $\theta \colon K \to L$ is a 
continuous mapping between compact spaces $K$ and $L$, we write $\theta_*\mu$ for the pushforward 
of a measure $\mu \in \uC(K)'$, i.e.,
\begin{align*}
  \int_L f  \,\mathrm{d}\theta_*\mu = \int_K f \circ \theta \dmu   \quad   \text{for } f \in \uC(L).
\end{align*}
Moreover, we define the \emph{Koopman operator} $T_\theta \in \mathscr{L}(\uC(L),\uC(K))$ of 
$\theta$ by $T_\theta f \defeq f\circ \theta$ for $f \in \uC(L)$. For a topological dynamical 
system $(K,G,\phi)$, the mapping
\begin{align*}
  T_\phi \colon G \to \mathscr{L}(\uC(K)), \quad 
  g \mapsto T_{\phi_{g^{-1}}}
\end{align*}
is the \emph{Koopman representation} of $(K,G,\phi)$.

\textbf{Acknowledgements.}
The authors express their gratitude towards Markus Haase, 
Rainer Nagel, and Jean Renault for ideas and inspiring discussions.
The first author thanks the MPIM and both authors thank the MFO 
for providing a stimulating atmosphere for working on this project.
The second author was supported by a scholarship of the Friedrich-Ebert-Stiftung 
while working on this article. Moreover, a preliminary version of this 
article is included in the second author's PhD thesis \cite{Krei2019}.

\section{Groupoids and groupoid actions}
\label{sec:groupoids}


Following \cite[Definitions 2.1, 2.2, and 2.17]{2013MitreaMonniaux}, we recall 
the definition of groupoids and semigroupoids which generalize the concepts of 
groups and semigroups, respectively, in the sense that the multiplication is 
only partially defined. The reader is also referred to \cite{IbortRodriguez2019}
or \cite{Higgins1971} as other general introductions to groupoids.

\begin{definition}\label{def:semigroupoid}
  A \emph{semigroupoid} consists of a set $\euS$, a set $\euS^{(2)} \subset \euS\times\euS$ of 
  composable pairs, and a product map $\pardot\colon \euS^{(2)} \to \euS$
  that is associative in the sense that
  \begin{enumerate}[(i)]
    \item if $(\pzg_1,\pzg_2),(\pzg_2,\pzg_3) \in \euS^{(2)}$, then 
    $(\pzg_1\cdot \pzg_2,\pzg_3),(\pzg_1,\pzg_2 \cdot \pzg_3) \in \euS^{(2)}$ and 
    $(\pzg_1\cdot \pzg_2)\cdot \pzg_3 = \pzg_1 \cdot (\pzg_2 \cdot \pzg_3)$.
  \end{enumerate}
  We usually abbreviate $\pzg\cdot\pzh$ by $\pzg\pzh$ if there is no room for confusion. 
  We call a semigroupoid $\euG$ a \emph{groupoid} if there is an \emph{inverse map} 
  ${}^{-1} \colon \euG \to \euG$
  such that, additionally, for each $\pzg\in \euG$
  \begin{enumerate}[resume*]
    \item $(\pzg^{-1}, \pzg) \in \euG^{(2)}$ and if $(\pzg, \pzh) \in \euG^{(2)}$, then $\pzg^{-1}(\pzg\pzh) = \pzh$,
    \item $(\pzg, \pzg^{-1}) \in \euG^{(2)}$ and if $(\pzh, \pzg) \in \euG^{(2)}$, then $(\pzh\pzg)\pzg^{-1} = \pzh$.
  \end{enumerate}
  If $\euG$ is a groupoid, 
  \begin{align*}
    \euG^{(0)} \defeq \left\{\pzg^{-1}\pzg \mmid \pzg\in \euG \right\}
  \end{align*}
  is called the \emph{unit space} of $\euG$ and the maps
  \begin{align*}
    s\colon \euG \to \euG^{(0)}, \quad  &   \pzg \mapsto \pzg^{-1}\pzg, \\
    r\colon \euG \to \euG^{(0)}, \quad  &   \pzg \mapsto \pzg\pzg^{-1}
  \end{align*}
  are called the \emph{source} and \emph{range maps} of $\euG$. For $\pzu, \pzv\in \euG^{(0)}$, we write 
  $\euG_\pzu \defeq s^{-1}(\pzu)$, $\euG^\pzv \defeq r^{-1}(\pzv)$, and $\euG_\pzu^\pzv \defeq \euG_\pzu\cap\euG^\pzv$. A 
  groupoid is \emph{transitive} if $\euG_\pzu^\pzv \neq \emptyset$ for all $\pzu,\pzv \in \euG^{(0)}$ and 
  \emph{a group bundle} if $\euG_\pzu^\pzv = \emptyset$ for all $\pzu,\pzv \in \euG^{(0)}$ with $\pzu \neq \pzv$. 
  If $\euG$ is a group bundle, we write $p \defeq r = s$. 
  \emph{Subsemigroupoids} and \emph{subgroupoids} of a given semigroupoid or 
  groupoid are defined in a straightforward way.
  
   A \emph{topological semigroupoid} 
  is a semigroupoid $(\euS, \euS^{(2)}, \cdot)$ 
  with a Hausdorff topology on $\euS$ such that the product map is continuous. We define \emph{topological 
  groupoids} analogously by demanding that the inverse map be continuous, too.
\end{definition}

Below, we collect examples of semigroupoids and groupoids which play an important role throughout the 
article (see also \cite[Examples 1.2]{Rena1980} for some of these and other examples).

\begin{example}\label{ex:trivialgrpd}
	Let $K$ be a compact space. Then $K$ is a compact groupoid with
		\begin{align*}
			K^{(2)} = \{(x,x) \mid x \in K\}
		\end{align*}
	and multiplication and inversion trivially defined. We call such a groupoid a \emph{trivial groupoid}.
\end{example}

\begin{example}
  Given a groupoid $\euG$, the subgroupoid
  \begin{align*}
    \Iso(\euG) \defeq \{ \pzg\in \euG \mid s(\pzg) = r(\pzg) \}
  \end{align*}
  of $\euG$ is a group bundle called the \emph{isotropy bundle} of $\euG$.
\end{example}

\begin{example}\label{ex:pairgroupoid}
  Let $K$ be a set. Then $\euG_K \defeq K\times K$ is a groupoid with the set of
  composable pairs
  \begin{align*}
    \euG_K^{(2)} \defeq \left\{ \left((x, y), (y, z)\right) \mmid x, y, z\in K \right\},
  \end{align*}
  a product map defined by $(x, y)\cdot(y, z) \defeq (x, z)$, and the inverse map 
  $(x, y) \mapsto (y, x)$. The groupoid $\euG_K$ is called the \emph{pair groupoid} 
  of $K$. It the property that the equivalence relations on $K$ can be identified 
  with full subgroupoids of $\euG_K$. Here a subgroupoid is called \emph{full} if 
  it has the same unit space as its ambient groupoid. Note that a subgroupoid of 
  $K\times K$ is transitive if and only if it equals $K\times K$.
\end{example}

\begin{example}
  For a topological space $X$, consider the space 
  \begin{align*}
    \uP(X) \defeq \uC([0,1], X)
  \end{align*}
  of paths in $X$, and define $\pi_1(X)$ to be the quotient of $\uP(X)$ modulo homotopy with 
  fixed end points. Then $\pi_1(X)$ is a groupoid with respect to the product map given by 
  concatenation of paths, called the \emph{fundamental groupoid} of $X$ (cf.\ \cite[Chapter 6]{Brown2006}). 
  The source and range maps send an equivalence class $[\gamma]$ to the starting and end points 
  $\gamma(0)$ and $\gamma(1)$, respectively. Moreover, the units in $\pi_1(X)$ 
  are the equivalence classes of constant paths and so the unit space $\pi_1(X)^{(0)}$ may be
  identified with $X$. The isotropy groups $\pi_1(X)_x^x$ for $x\in X$ are precisely the usual
  fundamental groups $\pi_1(X, x)$.
  
  A fundamental groupoid $\pi_1(X)$ is transitive if and only if $X$ is path-connected and 
  such fundamental groupoids are archetypal examples for transitive groupoids. If $\pi_1(X)$ 
  is transitive, all isotropy groups $\pi_1(X)_x^x$, $\pi_1(X)_y^y$ are isomorphic via conjugation 
  by a path $\eta$ from $x$ to $y$,
  \begin{align*}
    c_\eta\colon\pi_1(X)_y^y \to \pi_1(X)_x^x, \quad [\gamma] \mapsto [\eta]^{-1}[\gamma][\eta].
  \end{align*}
  In the same way one sees that, in general, all isotropy groups of a transitive groupoid 
  are isomorphic. This explains the heuristic that transitive groupoids behave similarly
  to groups, which we will make repeated use of. However, this does not mean that the study of a 
  transitive groupoid can always be replaced with the study of a single isotropy group, as it 
  is frequently done for the fundamental groupoid. Isotropy groups contain only part of 
  the picture and, as we will see, the groupoid perspective emerges as the natural conceptual
  generalization of existing approaches for groups.
\end{example} 

The following standard construction allows to completely encode the dynamics of a group action 
within a groupoid and motivates part of the terminology around groupoids.

\begin{example}\label{ex:actiongrpd}
	Let $(K,G)$ be a topological dynamical system. Then the \emph{action groupoid} or 
	\emph{transformation groupoid} $G \ltimes K$ is the set $G \times K$ with
		\begin{align*}
			(G \ltimes K)^{(2)} \defeq \{((g_2,g_1x),(g_1,x)) \mid g_1,g_2 \in G, x\in K\} 
		\end{align*}
	and 
	\begin{alignat*}{4}
		\cdot \, &\colon (G \ltimes K)^{(2)} \to G \ltimes K, \quad  & ((g_2,g_1x),(g_1,x)) \mapsto & \, (g_2g_1,x),\\
		^{-1}\, &\colon G \ltimes K \to G \ltimes K, \quad & (g,x)  \mapsto & \, (g^{-1},gx).
	\end{alignat*}
	We identify its unit space
		\begin{align*}
			(G \ltimes K)^{(0)} = \{(1,x)\mid x \in K\}
		\end{align*}
	with $K$.
\end{example}

The \enquote{structure-preserving maps} between semigroupoids are the following.

\begin{definition}
  A \emph{homomorphism} $\Phi \colon \euS \to \euT$ of semigroupoids (or groupoids) 
  is a mapping $\Phi \colon \euS \to \euT$ satisfying $(\Phi(\pzg_1),\Phi(\pzg_2)) \in \euT^{(2)}$ 
  and $\Phi(\pzg_1\pzg_2) = \Phi(\pzg_1)\Phi(\pzg_2)$ for all $(\pzg_1,\pzg_2) \in \euS^{(2)}$. We 
  write $\Phi^{(0)}$ for the induced map $\euS^{(0)} \to \euT^{(0)}$. Moreover, we 
  call $\Phi$ a \emph{factor map} and $\euT$ a \emph{factor} of $\euS$ if $\Phi$ is surjective. 
  Homomorphisms and factors of topological semigroupoids are defined by additionally 
  requiring $\Phi$ to be continuous.
\end{definition}

\begin{example}\label{ex:orbitrelationmorphism}
  Let $\euG$ be a groupoid.
  \begin{enumerate}
    \item The inclusions of the unit space $\euG^{(0)}$ and the isotropy 
    bundle $\Iso(\euG)$ are groupoid homomorphisms.
    \item The map 
    \begin{align*}
      (r, s)\colon \euG \to \euG^{(0)}\times \euG^{(0)}, \quad \pzg \mapsto \left( r(\pzg), s(\pzg) \right)
    \end{align*}
    is a groupoid morphism between $\euG$ and the pair groupoid $\euG^{(0)}\times\euG^{(0)}$. 
    In particular, its image $R_\euG$ is an equivalence relation, called the \emph{orbit relation} on 
    $\euG^{(0)}$.
  \end{enumerate}
\end{example}

We now recall the definition of groupoid actions (or $\euG$-spaces, see, e.g., \cite[Section 2.1]{DeRe2000}).

\begin{definition}\label{definition:groupoidaction}
  Let $\euG$ be a topological groupoid. A \emph{groupoid action} of $\euG$ on a compact space $K$ is 
  a tuple $(K,q,\euG,\phi)$ with a continuous, open surjection $q \colon K \to \euG^{(0)}$
  and a continuous mapping
  \begin{align*}
    \varphi \colon \euG \times_{s,q} K \to K, \quad (\pzg,x) \mapsto \varphi_\pzg(x) \eqdef \pzg x
  \end{align*}
	such that 
  \begin{enumerate}[(i)]
    \item $q(\pzg x) = r(\pzg)$ for all $(\pzg,x) \in \euG \times_{s,q} K$,
    \item $(\pzg_1\pzg_2 )x= \pzg_1(\pzg_2 x)$ for all $(\pzg_1,\pzg_2) \in \euG^{(2)}$ and 
    $x \in K_{s(\pzg_2)}$,
    \item $\pzu x = x$ for all $\pzu \in \euG^{(0)}$ and $x \in K_{\pzu}$.
  \end{enumerate}
  If $x\in K$, the \emph{orbit} of $x$ under $\euG$ is defined as 
  $\euG x \defeq \{\pzg x \mid \pzg \in \euG_{q(x)}\}$.
	A groupoid action is called
  \begin{itemize}
    \item \emph{transitive} if $\euG x = K$ for some/every $x \in K$.
    \item \emph{fiberwise transitive} if the fiber groups $\euG_{\pzu}^{\pzu}$ act transitively 
    on $K_{\pzu}$ for every $\pzu \in \euG^{(0)}$.
  \end{itemize}
	A \emph{morphism} $(p, \Phi) \colon (K_1,q_1,\euG_1) \to  (K_2,q_2,\euG_2)$ of groupoid actions consists
	of a groupoid morphism $\Phi\colon \euG_1 \to \euG_2$ and an 
	open continuous map $p \colon K_1 \to K_2$ such that
  \begin{enumerate}[(i)]
    \item $q_2 \circ p = \Phi^{(0)}\circ q_1$,
    \item $p(\pzg x) = \Phi(\pzg) p(x)$ for all $(\pzg,x) \in \euG \times_{s,q_1} K_1$.
  \end{enumerate}
  If $\euG_1 = \euG_2$ and $\Phi$ is the identity, we abbreviate $(p, \Phi)$ as $p$.
	A morphism $(p, \Phi)$ is called a \emph{factor map} or an \emph{extension} if $p$ and $\Phi$ are surjective.
	In this case,
	$(K_1, q_1, \euG_1)$ is called an \emph{extension} of $(K_2, q_2, \euG_2)$ and 
	$(K_2, q_2, \euG_2)$ a \emph{factor} of $(K_1, q_1, \euG_1)$.
\end{definition}

\begin{remark}\label{rem:transitive}
  As in the case of group actions, we usually omit $\phi$ and just write $(K,q,\euG)$ 
  for a groupoid action $(K, q, \euG, \phi)$.
  We emphasize that for a groupoid action $(K,q,\euG)$ we always require $K$ to be compact and 
  $q$ to be open. This implies, in particular, that the unit space $\euG^{(0)}$ is compact, too. 
  Note also that a groupoid action $(K, q, \euG)$ is transitive if and only if 
  $(K, q, \euG)$ is fiberwise transitive and $\euG$ is transitive.
\end{remark}

\begin{example}\label{ex:grpaction}
  Let $(K,G,\varphi)$ be a topological dynamical system. Then $(K,q,G,\phi)$ is a groupoid action 
  where $q \colon K \to \{1\} = G^{(0)}, \, x \mapsto 1$.
\end{example}

Next, we consider one of this article's key examples which motivates our systematic study 
of groupoid actions.

\begin{example}\label{ex:extensiongrpdaction}
  Let $q \colon (K, G, \phi) \to (L, G, \psi)$ be an open extension of topological dynamical systems. 
  Then the action groupoid $G \ltimes L$ defines a groupoid action $(K, q, G \ltimes L, \eta_\phi)$ via
  \begin{align*}
      \eta_\phi\colon(G \ltimes L)\times_{s,q} K \to K, \quad ((g, l), x) \mapsto \phi_g(x).
  \end{align*}
  Conversely, let $(L, G, \psi)$ be a topological dynamical system and $(K, q, G \ltimes L, \eta)$ 
  be an action of $G \ltimes L$ on $K$. Then it is not hard to check that 
  \begin{align*}
    \phi_\eta\colon G\times K \to K, \quad (g, x) \mapsto \eta_{(g, q(x))}(x)
  \end{align*}
  defines a continuous action $(K, G, \phi_\eta)$ of $G$ on $K$ such that 
  $q\colon (K, G, \phi_\eta) \to (L, G, \psi)$ is an extension of topological dynamical systems.
  Since these constructions are mutually inverse, an extension $q\colon (K, G, \phi) \to (L, G, \psi)$ 
  can be equivalently regarded as a groupoid action $(K, q, G \ltimes L, \eta)$ of the 
  action groupoid $G\ltimes L$ on $K$. In what follows, the reader should always have this example in mind when thinking
  about groupoid actions.
\end{example}

\begin{remark}
  There is also another way to obtain a groupoid action $(K, q, \euG)$ 
  from an extension $q\colon (K, G,\phi) \to (L, G,\psi)$ of dynamical systems:
  Let 
  \begin{align*}
   \euS(q) \defeq \left\{\phi_g|_{K_l}\colon K_l \rightarrow K_{gl} \mmid g\in G, l\in L\right\}.
  \end{align*}
  Then one can also study the groupoid action $(K, q, \euS(q))$. 
  Note, however, that there is a loss
  of information because the action of $G$ on $K$ can no longer be reconstructed 
  from $\euS(q)$ as opposed to the action of $G\ltimes L$ which, in addition,
  keeps track of which element of $G$ acts. This difference is often 
  immaterial, so there is no harm in thinking directly 
  about the \emph{transition groupoid} $\euS(q)$ instead of the action groupoid $G\ltimes L$.
\end{remark}

\begin{example}\label{ex:unitspaceaction}
	Every topological groupoid (with compact unit space) acts canonically on its unit space via 
	conjugation.
	Indeed, if $\euG$ is a groupoid with compact unit space, then $(\euG^{(0)},\mathrm{id}_{\euG^{(0)}}, \euG, \phi)$ with
		\begin{align*}
			\phi\colon\euG \times_{s,\mathrm{id}_{\euG^{(0)}}}\euG^{(0)} \to \euG^{(0)},\quad (\pzg,\pzu) \to \pzg \pzu \pzg^{-1}
		\end{align*}
	is a groupoid action mapping $s(\pzg)$ to $r(\pzg)$ for every $\pzg \in \euG$. It is the smallest 
	action of $\euG$ in the sense that if $(K,q,\euG)$ is 
	any groupoid action, then $q \colon K \to \euG^{(0)}$ defines an extension 
	$q \colon (\euG,q,K) \to (\euG,\mathrm{id}_{\euG^{(0)}}, \euG^{(0)})$ of groupoid actions. 
	Moreover, $(\euG^{(0)},\mathrm{id}_{\euG^{(0)}}, \euG)$ is always fiberwise transitive, 
	and it is transitive if and only if $\euG$ is a transitive groupoid.
\end{example}

As noted in \cref{ex:extensiongrpdaction}, an extension of topological 
dynamical systems can be equivalently regarded as a certain groupoid action,
suggesting that properties of extensions may be rephrased in terms 
of groupoid actions. The following definition carries out this straightforward 
translation to general groupoid actions for the standard notions of structuredness
of extensions (cf. \cite[Sections V.2 and V.5]{deVr1993} and \cite{Ausl2013}).

\begin{definition}\label{def:struct_ext}
  A groupoid action $(K,q,\euG)$ is called
  \begin{enumerate}[(i)]
    \item\label{def:struct_ext_item:weaklyequi} \emph{weakly equicontinuous} or \emph{stable} if 
    for each $\pzu \in \euG^{(0)}$ and each entourage $U\in \mathcal{U}_K$ there is an entourage 
    $V\in\mathcal{U}_K$ such that $(\pzg x, \pzg y)\in U$
    for all $\pzg \in \euG_\pzu$ and $x,y\in K_\pzu$ with $(x,y)\in V$.
    \item\label{def:struct_ext_item:equi} \emph{equicontinuous} if 
    for each entourage $U\in \mathcal{U}_K$ there is an entourage 
    $V\in\mathcal{U}_K$ such that for each $\pzu\in\euG^{(0)}$ one has $(\pzg x, \pzg y)\in U$
    for all $\pzg \in \euG_\pzu$ and $x,y\in K_\pzu$ with $(x,y)\in V$.
    \item\label{def:struct_ext_item:pseudoiso} \emph{pseudoisometric} if  there is a set $P$ of continuous mappings 
    $p\colon K \times_{q} K \to [0,\infty)$ such that
      \begin{itemize}
        \item $p_\pzu = p|_{K_\pzu \times K_\pzu}$ is a pseudometric on $K_\pzu$ for every $\pzu \in \euG^{(0)}$,
        \item the pseudometrics $p_\pzu$ for $p \in P$ generate the topology of $K_\pzu$ for every $\pzu \in \euG^{(0)}$,
        \item $p(\pzg x,\pzg y) = p(x,y)$ for all $x,y \in K_{s(\pzg)}$ and $\pzg \in \euG$.
      \end{itemize}
    \item\label{def:struct_ext_item:iso} \emph{isometric} if it is pseudoisometric and the set $P$ 
    can be chosen to consist of a single map which is (necessarily) a metric on each fiber.
  \end{enumerate}
\end{definition}

\begin{remark}\mbox{}
  \begin{enumerate}
    \item It is immediate from \cref{def:struct_ext} that if $q \colon (K,G) \to (L,G)$ is an 
    extension of topological dynamical systems, the extension is weakly equicontinuous, equicontinuous, \dots 
    if and only if the corresponding groupoid action $(K, q, G\ltimes L)$
    is. 
    \item We show in \cref{prop:pseudo_equi} below that a pseudoisometric groupoid action is 
    equicontinuous. Hence, 
    \ref{def:struct_ext_item:iso} $\implies$ \ref{def:struct_ext_item:pseudoiso}
    $\implies$
    \ref{def:struct_ext_item:equi} $\implies$ \ref{def:struct_ext_item:weaklyequi}. In general, none
    of the converse implications hold: For \ref{def:struct_ext_item:pseudoiso} and 
    \ref{def:struct_ext_item:iso}  this is obvious, for \ref{def:struct_ext_item:equi} and 
    \ref{def:struct_ext_item:pseudoiso} see \cref{ex:equinoniso} below, and for the relation between 
    and \ref{def:struct_ext_item:weaklyequi} and \ref{def:struct_ext_item:equi} we refer to \cite{Ausl2013}.
    \item Recall that if $(K,G)$ and $(L,G)$ are \emph{minimal} group actions, $q$ is 
    equicontinuous if and only if it is pseudoisometric if and only if it is weakly equicontinuous and 
    open, see \cite[Corollary 5.10]{deVr1993} and \cite[Theorem 3.13.17]{Bron1979}.
  \end{enumerate}
\end{remark}

\begin{proposition}\label{prop:pseudo_equi}
	Let $(K, q, \euG)$ be a pseudoisometric groupoid action. Then $(K, q, \euG)$ is equicontinuous.
\end{proposition}

\begin{proof}
	Pick a set $P$ as in \cref{def:struct_ext} \ref{def:struct_ext_item:pseudoiso}. For each finite subset $F \subset  P$ and $\epsilon > 0$, set 
  \begin{align*}
    U_{F,\epsilon} \defeq \left\{(x,y) \in K \times_q K \mmid \forall p\in F\colon p(x,y) < \epsilon \right\}
  \end{align*}
  and note that 
	\begin{align*}
    \bigcap_{\substack{F\subset P \text{ finite}\\ \epsilon > 0}} U_{F, \epsilon} = \Delta_K.
  \end{align*}
  We claim that for every $U\in \mathcal{U}_K$, there are a finite set 
  $F\subset P$ and an $\epsilon > 0$ such that $U_{F, \epsilon} \subset U$ which would yield 
  the claim since $U_{F, \epsilon}$ is $\euG$-invariant. In order to prove the claim, first recall 
  that $\mathcal{U}_K = \mathcal{U}_{K\times K}(\Delta_K)$ is just the neighborhood filter 
  of the diagonal. The claim then follows from the fact that if $(M_\alpha)_{\alpha\in A}$ 
  is a decreasing family of sets in a compact space $X$ and $U$ is an open neighborhood of 
  $\bigcap_{\alpha\in A} \overline{M_\alpha}$, then there is an $\alpha_0 \in A$ such that 
  $M_{\alpha_0} \subset U$ (use the finite intersection property).
\end{proof}

In analogy to group actions, we can associate an
action groupoid to a groupoid action (cf.\ \cite[Section 2.1]{DeRe2000}). We note this construction 
for later reference since it will allow to investigate the orbit structure of groupoid actions.

\begin{definition}\label{def:actiongrpd2}
  For $(K,q,\euG)$ a groupoid action we define the \emph{action groupoid $\euG \ltimes K$ of 
  $(K,q,\euG)$} as the fiber product $\euG \times_{s,q} K$ with
  \begin{align*}
    (\euG \ltimes K)^{(2)} \defeq 
    \left\{\left(\left(\pzh,\pzg x\right), \left( \pzg, x \right)\right) 
    \mmid x \in K, \pzg \in \euG_{q(x)}, \pzh \in \euG_{r(\pzg)} \right\}
  \end{align*}
  and the operations
  \begin{align*}
    \cdot &\colon (\euG \ltimes K)^{(2)} \to \euG \ltimes K, \quad 
    \left(\left(\pzh,\pzg x\right),\left(\pzg,x\right)\right) \mapsto \left(\pzh\pzg, x\right),\\
    ^{-1} &\colon \euG \ltimes K \to \euG \ltimes K, \quad 
    \left(\pzg,x\right) \mapsto \left(\pzg^{-1},\pzg x\right).
  \end{align*}
  We identify its unit space
  \begin{align*}
    (\euG \ltimes K)^{(0)} = \{(q(x),x) \mid x \in K\}
  \end{align*}
  with $K$.
\end{definition}

\section{The compact-open topology for fiber maps}
\label{sec:topology}

In order to define a uniform enveloping semigroupoid for groupoid actions, 
it is necessary to find an appropriate topological space in which to carry out the 
compactification. For the uniform enveloping semigroup of a topological dynamical system $(K,G)$, this is the space $\uC(K, K)$ 
endowed with the compact-open topology, as explained in the introduction. To generalize this, 
we extend the compact-open topology to \enquote{fibered mappings} in this section and then 
introduce the uniform enveloping semigroup in \cref{sec:enveloping_semigroupoids}.

\begin{definition}\label{def:contfiber}
  For topological spaces $X$ and $Y$ as well as continuous surjections $p \colon X \to L$ and $q\colon Y\to L'$ onto compact spaces $L$ and $L'$ we set
  \begin{align*}
      \uC_p^q(X,Y)_{l}^{l'} & \defeq \mathrm{C}\left(X_{l},Y_{l'}\right) \quad  \text{for } (l, l')\in L\times L'
  \end{align*}
  and define the set of \emph{continuous fiber maps} between $p\colon X\to L$ and $q\colon 
  Y\to L'$ as
  \begin{align*}
    \uC_{p}^{q}(X,Y) \defeq \bigcup_{l \in L, l'\in L'} \uC_{p}^{q}(K,X)_{l}^{l'}.
  \end{align*}
  
  We define \enquote{source} and \enquote{range} maps 
  \begin{align*}
    s\colon \uC_p^q(X,Y) \to L, \quad r\colon \uC_p^q(X,Y) \to L'
  \end{align*}
  by setting 
  \begin{align*}
    s(\theta) \defeq l  \quad 
    \text{and}          \quad 
    r(\theta) \defeq l' \quad  
    \text{for } \theta \in \uC_p^q(K,X)_{l}^{l'}.
  \end{align*}
  If $Y$ is a topological space and $q\colon Y\to \pt$ is the unique map onto a 
  one-point space $\pt$, we abbreviate $\uC_{p}(X, Y) \defeq \uC_{p}^{q}(X,Y)$.
  Moreover, we write $\uC_{p}(X) \defeq \uC_{p}(X,\C)$.
\end{definition}

\begin{remark}\label{rem:fibermapgroupod}
  If $L = L'$ and $p=q$ in the definition above, the set $\uC_q^q(K, K)$ with
  \begin{align*}
    \uC_q^q(K, K)^{(2)}
    \defeq \{(\theta, \eta) \in \uC_q^q(K, K)\times\uC_q^q(K, K) \mid r(\theta) = s(\theta)\}
  \end{align*}
  is a semigroupoid with composition of mappings as the product map.  We call this the \emph{semigroupoid of
 continuous fiber maps} of $q$. 
\end{remark}

The following generalization of the compact-open topology for spaces of fiber maps is 
taken from \cite[Section 1]{BoothBrown1978} where it is considered on the larger set of all 
partial maps.

\begin{definition}
  Let $p \colon X \to L$ and $q\colon Y\to L'$ be
  continuous surjections of topological spaces $X$ and $Y$ onto compact 
  spaces $L$ and $L'$. For a compact subset $C \subset X$ and an open subset $O\subset Y$,
  set 
  \begin{align*}
    \uW(C, O) \defeq \left\{ \theta\in \uC_p^q(X, Y) \mmid \theta(C) \subset O \right\} \subset \uC_p^q(X, Y)
  \end{align*}
  where $\theta(C) = \{\theta(x) \mid x\in C \textrm{ with } p(x) = s(\vartheta)\}$.
  We then define the
  \emph{compact-open topology} on $\uC_p^q(X,Y)$ to be the topology generated 
  by all the sets of the form $\uW(C, O)$.
\end{definition}

The classical characterization of convergence with respect to the compact-open topology 
for locally compact spaces readily extends to this more general context.

\begin{proposition}\label{charconv}
  Let $p\colon X \to L$, $q\colon Y\to L'$ be continuous surjections of topological spaces $X$ and $Y$ onto compact spaces $L$ and $L'$. Suppose that $X$ is locally 
  compact. Then for a net $(\theta_{\alpha})_{\alpha \in A}$ in $\uC_p^q(X,Y)$ and a
  $\theta \in \uC_p^q(X,Y)$ the following assertions are equivalent.
  \begin{enumerate}[(a)]
    \item\label{charconv_co} $\lim_\alpha \theta_\alpha = \theta$ with respect to 
          the compact-open topology.
    \item\label{charconv_nets} The following two conditions are satisfied.
      \begin{itemize}
        \item $\lim_{\alpha} s(\theta_\alpha) = s(\theta)$.
        \item If $(\theta_\beta)_{\beta \in B}$ is a subnet of $(\theta_\alpha)_{\alpha \in A}$, then
        \begin{align*}
          \lim_\beta \theta_\beta(x_\beta) = \theta(x)
        \end{align*}	
        for every net $(x_\beta)_{\beta \in B}$ in $X$ that converges to some $x \in X$ and 
        satisfies $q(x_\beta) = s(\theta_\beta)$ for every $\beta \in B$.   			
        \end{itemize}
  \end{enumerate}
  In particular, the compact-open topology is the coarsest topology on 
  $\uC_p^q(X,Y)$ such that the maps
  \begin{alignat*}{3}
  	s & \colon \uC_p^q(X,Y) \to L, \quad &\theta &\mapsto s(\theta)\\
    \operatorname{ev}&\colon \uC_p^q(X,Y)\times_{s,p} X \to Y, \quad &(\theta, x) &\mapsto \theta(x)
  \end{alignat*}
  are continuous.
\end{proposition}
\begin{proof}
  First, suppose that $(\theta_\alpha)_{\alpha\in A}$ and $\theta$ satisfy \ref{charconv_nets}
  and suppose that $(\theta_\alpha)_{\alpha \in A}$ does not converge to $\theta$. Then there 
  are a compact set $C\subset X$ and an open set $O\subset Y$ such that $\theta\in\uW(C, O)$ 
  and such that $(\theta_\alpha)_{\alpha\in A}$ does not eventually lie in $\uW(C, O)$. This means
  that there is a subnet $(\theta_\beta)_{\beta \in B}$ of $(\theta_\alpha)_{\alpha\in A}$ satisfying
  $\theta_\beta\not\in \uW(C, O)$ for each $\beta \in B$. After again passing to a subnet, we may thus 
  assume that there is a convergent net $(x_\beta)_{\beta\in B}$ in $C$ such that $x_\beta\in X_{s(\theta_\beta)}$
  and $\theta_\beta(x_\beta)\not\in O$ for each $\beta\in B$. However, by \ref{charconv_nets} 
  $\lim_\beta \theta_{\beta}(x_\beta) = \theta(x) \in O$, a contradiction. Hence, \ref{charconv_nets}
  $\implies$ \ref{charconv_co}.
  
  Now, suppose that $(\theta_\alpha)_{\alpha\in A}$ is a net in $\uC_p^q(X, Y)$ 
  converging to $\theta\in\uC_p^q(X, Y)$ in the compact-open topology. To see that 
  $\lim_\alpha s(\theta_\alpha) = s(\theta)$, suppose
  that $s(\theta_\alpha) \not\to s(\theta)$. We construct a compact set $C\subset X$ such that 
  $\theta\in \uW(C, \emptyset)$ and a subnet of $(\theta_\alpha)_{\alpha\in A}$ which avoids
  $\uW(C, \emptyset)$. To do this, first use the compactness of $L$ to find a 
  subnet $(\theta_\beta)_{\beta\in B}$
  of $(\theta_\alpha)_{\alpha\in A}$ such that $(s(\theta_\beta))_{\beta \in B}$ converges
  to another point than $s(\theta)$. Moreover, by again passing to a subnet, one may 
  assume that
  there is a net $(x_\beta)_{\beta\in B}$ in $X$ converging to some $x\in X$ such that 
  $x_\beta \in X_{s(\theta_\beta)}$ for each $\beta \in B$. Since by assumption, $p(x_\beta) = s(\theta_\beta)$ converges to a point 
  different from $s(\theta)$, $p(x) \neq s(\theta)$. Therefore, there is a compact neighborhood 
  $C\in\mathcal{U}(x)$ with $C\cap X_{s(\theta)} = \emptyset$. Now $\theta\in\uW(C, \emptyset)$
  whereas eventually $\theta_\beta\not\in \uw(C, \emptyset)$ since $x_\beta \to x$, a contradiction.
  Therefore, $\lim_\alpha s(\theta_\alpha) = s(\theta)$.
  
  To establish the second part of \ref{charconv_nets}, let $(\theta_\beta)_{\beta\in B}$
  and $(x_\beta)_{\beta\in B}$ be as in \ref{charconv_nets}. Let $O\in\mathcal{U}(\theta(x))$
  be a neighborhood of $\theta(x)$. Then $V \defeq \theta^{-1}(O)$ is a neighborhood of $x$
  with respect to the subspace topology on $K_{s(\theta)}$. Therefore, there is a compact neighborhood
  $C\in\mathcal{U}(x)$ such that $C\cap K_{s(\theta)} \subset V$. Because $x_\beta \to x$, 
  $(x_\beta)_{\beta\in B}$ eventually lies in $C$ and since $\theta_\beta \to \theta$ in the 
  compact-open topology, we conclude that $(\theta_\beta)_{\beta\in B}$ eventually lies in 
  $\uW(C, O)$. Since $O$ was an arbitrary neighborhood of $\theta(x)$, this shows that 
  $\lim_\beta \theta_\beta(x_\beta) = \theta(x)$.
\end{proof}

\begin{remark}\label{rem:quotientmapopenchar}
  In general, the compact-open topology on $\uC_p^q(X, Y)$ is not Hausdorff. In fact, it 
  is not difficult to infer from the characterization in \cref{charconv} that it is Hausdorff if 
  and only if $p$ is open, as it will henceforth always be the case. To show that the compact-open limit 
  $\theta$ of a net $(\theta_\alpha)_{\alpha\in A}$
  is unique, it suffices to show for every $x\in X_{s(\theta)}$ that $\theta(x)$ is uniquely determined.
  To see that openness implies this, one can use the observation that a continuous surjection 
  $p\colon X \to L$ between a locally compact space $X$ and 
  a compact space $L$ is open if and only if the following condition is fullfilled: For every convergent net 
  $(l_\alpha)_{\alpha \in A}$ in $L$ with limit $l\in L$ and every $x\in X_l$, there are a subnet 
  $(l_\beta)_{\beta \in B}$ of $(l_\alpha)_{\alpha \in A}$ and a net $(x_\beta)_{\beta \in B}$ 
  in $X$ that converges to $x$ and
  covers $(l_\beta)_{\beta \in B}$ in the sense that $p(x_\beta) = l_\beta$ for every $\beta \in B$.
  We will make use of this observation at several more occasions.
\end{remark}

In order to prove a generalization of the Arzel\`a-Ascoli theorem below in \cref{arzelaascoli},
we will need an equivalent description of the compact-open topology. To find another natural 
way to topologize $\uC_p^q(X, Y)$, observe that 
an element $\theta \in \uC_{p}^{q}(X, Y)$ may be identified with its graph 
$\Gr(\theta) \subset X\times Y$. Therefore, $\uC_{p}^{q}(X,Y)$ may be regarded as 
a subspace of the space $\mathscr{C}(X\times Y)$ of closed subsets of $X\times Y$, 
on which there exist many topologies, e.g., the Vietoris topology.

\begin{definition}
  Let $X$ be a topological space and $\mathscr{C}(X)$ the set of its nonempty closed subsets. 
  The \emph{Vietoris topology} on $\mathscr{C}(X)$ is the topology generated by the sets
    \begin{align*}
      &U^- \defeq \left\{A \in \mathscr{C}(X) \mid A \cap U \neq \emptyset\right\},\\
      &U^+ \defeq \left\{A \in \mathscr{C}(X) \mid A \subset U \right\}
    \end{align*}
  for open subsets $U \subset X$.
\end{definition}

\begin{remark}\label{compactvietoris}
  It is known that if $X$ is a Hausdorff space, then so is $\mathscr{C}(X)$, see \cite[Theorem 4.9]{Mich1951}. 
  If $X$ is compact, then $\mathscr{C}(X)$ is also compact, see \cite[Theorem 4.9]{Mich1951} or 
  \cite[Proposition 5.A.3]{ElEl2014}. If, additionally, 
  $X$ is a metric space, the Vietoris topology coincides with the topology induced by
  the Hausdorff metric, see \cite[Theorem 3.4 and Proposition 3.6]{Mich1951} or \cite[Exercise 5.4]{ElEl2014}.
\end{remark}

\begin{definition}\label{def:comopen}
  If $p \colon X \to L$ and $q\colon Y\to L'$ are
  continuous surjections of topological spaces $X$ and $Y$ onto compact spaces $L$ and $L'$, we define the
  \emph{Vietoris topology} on $\uC_p^q(X,Y)$ to be 
  the initial topology with respect to the map
  \begin{align*}
    \operatorname{Gr}\colon\uC_p^q(X,Y) \to \mathscr{C}(X\times Y), \quad 
    \theta \mapsto \operatorname{Gr}(\theta)
  \end{align*}
  where $\mathscr{C}(X\times Y)$ is equipped with the Vietoris topology.
\end{definition}

\begin{remark}
  In \cite{HolaZsilinszky2010}, a slightly different version of the compact-open
  topology on $\uC_p^q(X, Y)$ is considered which, additionally, uses nonempty open sets $U\subset X$ 
  and adds all the sets of the form
  \begin{align*}
    [U] \defeq \left\{ f\in \uC_p^q(X, Y) \mmid f^{-1}(Y)\cap U \neq \emptyset \right\}
  \end{align*}
  to generate the topology. If $X$ is compact and $p$ is open, however, $p^{-1}(p(U))^c \subset X$ is compact
  and $[U] = \uW(p^{-1}(p(U))^c, \emptyset)$. Therefore, if $X$ is compact, our compact-open
  topology coincides with the one considered in \cite{HolaZsilinszky2010}. 
  In particular, we can note the following theorem which is formulated more generally in 
  \cite[Proposition 2.2]{HolaZsilinszky2010} for later use.
\end{remark}

\begin{theorem}
  Let $p \colon X \to L$ and $q\colon Y\to L'$ be
  continuous surjections of topological spaces $X$ and $Y$ onto compact spaces $L$ and $L'$.
  If $X$ and $Y$ are compact and $p$ is open, then the Vietoris topology and the compact-open topology on 
  $\uC_p^q(X, Y)$ coincide.
\end{theorem}

\section{Uniform enveloping semigroupoids}
\label{sec:enveloping_semigroupoids}

We are now ready to introduce the uniform enveloping semigroupoid
of a groupoid action. The main result of this section is \cref{thm:pseudoisochar}
which states that, under appropriate assumptions, $\euE_\uu(K,q,\euG)$ is compact 
if and only if the groupoid action $(K,q,\euG)$ is pseudoisometric.

To define uniform enveloping semigroupoids, we regard the semigroupoid $\uC_q^q(K, K)$ of fiber maps 
introduced in \cref{rem:fibermapgroupod} as a topological semigroupoid with respect to the 
compact-open topology. This allows to define the uniform enveloping semigroupoid 
of a set of fiber maps.

\begin{definition}
  Let $q\colon K \to L$ be an open, continuous surjection of compact spaces and
  let $\euF$ be a subset of the topological semigroupoid $\uC_q^q(K, K)$. Then 
  the \emph{uniform enveloping semigroupoid} $\euE_\uu(\euF)$ of 
  $\euF$ is defined to be the smallest closed subsemigroupoid of $\uC_q^q(K, K)$
  containing $\euF$.
\end{definition}

\begin{remark}
  Note that this definition makes sense since the intersection of a family of closed subsemigroupoids 
  of a topological semigroupoid is again a closed subsemigroupoid.
\end{remark}

\begin{definition}\label{def:uniformenvsgrpd}
	Let $(K, q, \euG, \phi)$ be a groupoid action and consider the \emph{transition semigroupoid}
	$\euS(K, q, \euG, \phi)$ given by
		\begin{align*}
			\euS(K, q, \euG, \phi) \defeq 
			\left\{\varphi_\pzg \colon K_{s(\pzg)}\rightarrow K_{r(\pzg)} \mmid \pzg \in \euG \right\} 
			\subset \mathrm{C}_q^q(K,K).
		\end{align*}  
	We call 
	\begin{align*}
			 \euE_\uu(K, q, \euG, \phi) \defeq \euE_\uu(\euS(K, q, \euG, \phi))
	\end{align*}
	the \emph{uniform enveloping semigroupoid of $(K, q, \euG, \phi)$}.
\end{definition}

\begin{example}
  Let $(K, G, \phi)$ be a topological dynamical system and interpret it as a groupoid action 
  $(K, q, G, \phi)$ of $G$  (see \cref{ex:grpaction}). Then 
  \begin{align*}
    \euS(K, q, G, \phi) = \left\{ \phi_g \mid g\in G\right\} \subset \uC(K, K)
  \end{align*}
  is the transition group of $(K, G)$ and the enveloping semigroupoid
  \begin{align*}
    \euE_\uu(K, q, G, \phi) = \euE_\uu(\euS(K, q, G, \phi)) \subset \uC(K, K)
  \end{align*}
  is precisely the uniform enveloping semigroup $\uE_\uu(K, G)$. Therefore, $\euE_\uu$ 
  generalizes the uniform enveloping semigroup to arbitrary groupoid actions.
\end{example}

\begin{example}
  Let $q\colon (K, G, \phi) \to (L, G, \psi)$ be an open extension of topological dynamical 
  systems. As noted in \cref{ex:extensiongrpdaction}, we can equivalently regard
  the extension as an action $(K, q, G \ltimes L, \eta_\phi)$ of the action groupoid $G\ltimes L$.
  For this groupoid action, the transition groupoid is
  \begin{align*}
    \euS(K, q, G \ltimes L, \eta_\phi) = \left\{ \phi_g|_{K_l} \mmid g\in G, l \in L \right\} \subset \uC_q^q(K, K)
  \end{align*}
  and the uniform enveloping semigroupoid is
  \begin{align*}
    \euE_\uu(K, q, G \ltimes L, \eta_\phi) 
    = \euE_\uu\left(\left\{ \phi_g|_{K_l} \mmid g\in G, l \in L \right\}\right) 
    \subset \uC_q^q(K, K).
  \end{align*}
  We will use the notations $\euS(q)$ and $\euE_\uu(q)$ to abbreviate these (semi)groupoids.
\end{example}

\begin{example}\label{ex:disc}
  Consider the rotation on the disc with varying speed of rotation, i.e., the system 
  $(K, \phi)$ given by $K \defeq \D = \{ z\in \C \mid |z| \leq 1\}$ and
  \begin{align*}
    \phi\colon K\to K, \quad \phi(z) = \ue^{\ui |z|}z.
  \end{align*}
  If we set $(L, \psi) \defeq ([0,1], \id_{[0,1]})$, then 
  \begin{align*}
    q\colon (K, \phi) \to (L, \psi), \quad z \mapsto |z|
  \end{align*}
  defines an isometric extension between the two systems. If, for $l\in [0,1]$ and 
  $\alpha\in\T$, one lets 
  \begin{align*}
    \theta_{\alpha, l} \colon K_l \to K_l, \quad z \mapsto \alpha z
  \end{align*}
  be the rotation by $\alpha$ on $K_l$, then it is instructive to verify that
  \begin{align*}
    \euE_\uu(q) = \left\{ \vartheta_{\alpha, l} \mid \alpha \in \T, l\in L \right\}.
  \end{align*}
  In particular, $\euE_\uu(q)$ is a compact groupoid. This should be 
  contrasted with the much larger Ellis semigroup $\uE(K, G)$ of the system 
  which contains a homeomorphic copy of $\beta\N$ since $(K, G)$ is not tame (see, e.g., \cite[Theorem 1.2]{Glas2006a}).
\end{example}

\begin{example}\label{ex:skewgroupoid}
  Let $\alpha\in\T$ and
  \begin{alignat*}{3}
    \psi_\alpha \colon& \T \to \T, & \quad \psi_\alpha(x) &\defeq \alpha x, \\
    \phi_\alpha \colon& \T^2 \to \T^2, & \quad \phi_\alpha(x, y) &\defeq (\alpha x, xy)
  \end{alignat*}
  be the rotation by $\alpha$ and the corresponding skew rotation.
  Then $q\colon (\T^2, \psi_\alpha) \to (\T, \psi_\alpha)$, 
  $q(x, y) = x$ defines an isometric extension between the two systems. If $\alpha$ is rational,
  then $\phi_\alpha$ is periodic with some period $N\in\N$ and so
  \begin{align*}
    \euE_\uu(q) 
    = \euS(q) 
    = \left\{ \phi_\alpha^n|_{\T^2_l} \mmid l\in\T, n=1, \dots, N \right\}.
  \end{align*}
  Clearly, $\euE_\uu(q)$ is a groupoid. Moreover, for fixed $n\in\N$, it follows 
  from the characterization \cref{charconv} $\euE_\uu(q)$ that is compact. In case $\alpha$ is 
  irrational, 
  $(\T^2, \phi_\alpha)$ is minimal. Therefore, if one defines for $(\beta, \gamma)\in\T^2$
  \begin{align*}
    \theta_{(\beta, \gamma)}\colon\T^2 \to \T^2, \quad (x, y) \mapsto (\beta x, \gamma y),
  \end{align*}
  then 
  \begin{align*}
    \euE_\uu(q) = \left\{ \theta_{(\beta,\gamma)} \mmid (\beta,\gamma)\in\T^2 \right\}.
  \end{align*}
  In particular $\euE_\uu(q)$ is a compact groupoid.
\end{example}

\begin{example}\label{ex:envelopingrelation}
  Let $K$ be a compact space, consider the pair groupoid $K \times K$ and let 
  $R \subset K\times K$ be a full subgroupoid, i.e., an equivalence relation on $L$ 
  (see \cref{ex:pairgroupoid}). Then $\euE_\uu(R)$ is the smallest closed equivalence 
  relation on $K$ that contains $R$. Many important equivalence relations in topological 
  dynamics such as the equicontinuous structure relation or the distal structure relation
  arise in this way.
  
  As a special case, let $\euG$ be a topological groupoid
  with compact unit space and $(\euG^{(0)}, \id_{\euG^{(0)}}, \euG)$ be its action on its 
  unit space (see \cref{ex:unitspaceaction}). Then $\euE_\uu(\euG^{(0)}, \id_{\euG^{(0)}}, \euG)$ 
  lies in the groupoid
  \begin{align*}
    \uC_{\id_{\euG^{(0)}}}^{\id_{\euG^{(0)}}}(\euG^{(0)},\euG^{(0)})
  \end{align*}
  which is canonically isomorphic to the pair groupoid $\euG^{(0)}\times\euG^{(0)}$. Therefore,
  we can identify the groupoid $\euE_\uu(\euG^{(0)}, \id_{\euG^{(0)}}, \euG)$ with a closed equivalence
  relation on $\euG^{(0)}\times\euG^{(0)}$. This equivalence relation is given by
  $\euE_\uu(R_\euG)$ where $R_\euG \subset \euG^{(0)}\times\euG^{(0)}$ is the orbit relation
  $R_\euG = (r, s)(\euG) \subset \euG^{(0)}\times\euG^{(0)}$ 
  of the action of $\euG$ on its unit space, see \cref{ex:orbitrelationmorphism}.
  Therefore, $\euE_\uu(\euG^{(0)}, \id_{\euG^{(0)}}, \euG)$ can be identified with
  the smallest closed equivalence relation on $\euG^{(0)}\times\euG^{(0)}$ that contains the 
  orbit relation $R_\euG$.
\end{example}

\begin{remark}\label{rem:closurenotsgrpd}
  Note that the definition of the uniform enveloping semigroupoid $\euE_\uu(K,q,\euG)$ of a groupoid action $(K,q,\euG)$
  is more intricate than that of the uniform enveloping semigroup $\uE(K,G)$ of a group action $(K,G)$: The uniform 
  enveloping semigroup is defined as the closure of a semigroup and it turns out that this closure is automatically
  again a semigroup. In contrast to this, the following example demonstrates that $\euE_\uu(K, q, \euG)$ is
  generally not merely the closure of $\euS(K, q, \euG)$.
\end{remark}

\begin{example}\label{ex:twoxsquared}
  Consider the dynamical systems $(L, \psi)$ defined by $L \defeq [-1, 1]$, $\psi(x) \defeq \sign(x)x^2$ for $x \in L$
  and $(K, \phi)$ given by $K \defeq [-1, 1]\times \Z_2$, $\phi(x, g) \defeq (\psi(x), g+1)$ for $(x,g) \in K$.
  Then the map
  \begin{align*}
    q\colon (K, \phi) \to (L, \psi), \quad (x, g) \mapsto x 
  \end{align*}
  defines an isometric extension. The uniform enveloping semigroupoid 
  of $q$ is given by 
  \begin{align*}
    \euE_\uu(q) = \left\{ \theta_{x, y, h} \mmid x, y\in L, h\in \Z_2\right\}
  \end{align*}
  where $\theta_{x, y, h}$ denotes the function
  \begin{align*}
    \theta_{x, y, h} \colon K_x \to K_y, \quad (x, g) \mapsto (y, g + h).
  \end{align*}
  In contrast to this,
  \begin{align*}
    \overline{\euS(q)} = \euS(q) 
    &\cup 
    \left\{\theta_{x, 0, h}, \theta_{0,x, h} \mmid x \in [-1, 1], h\in \Z_2\right\} \\  
    &\cup  \left\{\theta_{-1, y, h}, \theta_{y, -1, h} \mmid y \in [-1,0], h\in \Z_2\right\} \\
    &\cup  \left\{\theta_{1, y, h}, \theta_{y,1, h} \mmid y \in [0,1], h\in \Z_2\right\}.
  \end{align*}
  Thus, the inclusion $\overline{\euS(q)} \subset \euE_\uu(q)$ is generally strict.
\end{example}

\subsection{Characterizing compactness.} Usually, the uniform enveloping semigroupoid 
is neither compact, nor a groupoid. We therefore try to answer the 
question: When is the uniform enveloping semigroupoid 
actually a compact groupoid?
As a first step to address this problem, we observe that the groupoid property
follows automatically once we have ensured compactness.

\begin{proposition}\label{compactimpliesgroupoid}
 Let $(K,q,\euG)$ be a groupoid action. If $\euE_\uu(K,q,\euG)$ is compact, then it is a compact groupoid, i.e., every 
  $\theta \in \euE_\uu(K,q,\euG)$ has an inverse $\theta^{-1} \in \euE_\uu(K,q,\euG)$ and the mapping 
  $\,^{-1}\colon \euE_\uu(K,q,\euG) \to \euE_\uu(K,q,\euG)$ is a homeomorphism.
\end{proposition}
\begin{proof}
  Consider the set $M$ of all elements $\theta \in \euE_\uu(K,q,\euG)$ having an inverse $\theta^{-1}$ in $\euE_\uu(K,q,\euG)$. 
  Then $M$ is certainly closed under compositions and contains $\euS(K,q,\euG)$. To see 
  that $M= \euE_\uu(K,q,\euG)$ it therefore 
  suffices to show that $M$ is closed in $\euE_\uu(K,q,\euG)$. Pick a net 
  $(\theta_{\alpha})_{\alpha \in A}$ in $M$ converging to $\theta \in \euE_\uu(K,q,\euG)$. 
  Passing to a subnet, we may assume that $(\theta_\alpha^{-1})_{\alpha \in A}$ converges 
  to some element $\varrho \in \euE_\uu(K,q,\euG)$. Using the characterization from \cref{charconv} 
  and the openness of $q$, we conclude that $\varrho = \theta^{-1}$. This shows that 
  $M = \euE_\uu(K,q,\euG)$. Moreover, if 
  $(\theta_{\alpha})_{\alpha \in A}$ is a net in $\euE_\uu(K,q,\euG)$ converging to some 
  $\theta \in \euE_\uu(K,q,\euG)$, then a similar argument shows that $\theta^{-1}$ is the only 
  cluster point of the net $(\theta_{\alpha}^{-1})_{\alpha \in A}$.
\end{proof}

We now try to characterize the compactness of 
the uniform enveloping semigroupoid by investigating when a set is (pre)compact in the 
compact-open topology. To this end, recall that 
if $K$ is a compact space and $Y$ is a uniform space, the precompactness of a subset 
$\mathcal{F} \subset \uC(K, Y)$ in the compact-open topology is characterized by the 
classical Arzel\`a-Ascoli theorem: $\mathcal{F}$ is precompact if and only if $\mathcal{F}$ is 
equicontinuous and $\im(\mathcal{F}) = \bigcup_{f\in\mathcal{F}} \im(f)$ is precompact in $Y$. In 
what follows, we generalize the notion of equicontinuity and the Arzel\`a-Ascoli theorem to 
compact bundles.

\begin{definition}
  Let $p\colon X \to L$, $q\colon Y\to L'$ be continuous surjections onto compact spaces
  and $X$ and $Y$ be uniform spaces. A subset
  $\euF \subset \uC_p^q(X,Y)$ is called \emph{(uniformly) equicontinuous}
  if for each $U \in \mathcal{U}_Y$ there is a $V \in \mathcal{U}_X$ such that  
  $(\theta(x_1),\theta(x_2)) \in U$ for every $\theta \in \euF$ and every 
  $(x_1,x_2) \in V \cap X \times_L X$ with $s(\vartheta) = p(x_1) = p(x_2)$.
\end{definition}

\begin{theorem}\label{arzelaascoli}
  Let $p\colon K \to L$, $q\colon Y\to L'$ be continuous surjections onto compact spaces,
  $K$ be compact, and $Y$ be a Hausdorff uniform space. If $p$ is open, then a subset 
  $\euF \subset \uC_p^q(K,Y)$ is precompact if 
  and only if the following two conditions are fulfilled.
  \begin{enumerate}[(i)]
    \item\label{arzelaascoli_precompact} $\im(\euF) \subset Y$ is precompact.
    \item\label{arzelaascoli_equi} $\euF$ is equicontinuous.
  \end{enumerate}
\end{theorem}
\begin{proof}
  Suppose that \ref{arzelaascoli_precompact} and \ref{arzelaascoli_equi} hold. In view of 
  \cref{compactvietoris}, it suffices to show that the 
  closure $\overline{\Gr(\euF)}$ in $\mathscr{C}(K \times Y)$ is in fact contained in 
  $\Gr(\uC_p^q(K,Y))$. So we pick $C\in \overline{\Gr(\euF)}$ and show that 
  $C = \Gr(\theta)$ for some $\theta\in\uC_p^q(K,Y)$.
  
  Let $(\theta_\alpha)_{\alpha\in A}$ be a net in $\euF$ such that $\Gr(\theta_\alpha) \to C$ 
  with respect to the Vietoris toplogy.
  First, let $(x, y) \in C$ and set $l \defeq p(x)$, $l'\defeq q(x)$. We claim that 
  $C \subset K_l\times Y_{l'}$: If $U\in\mathcal{U}_L(l)$ and $V\in\mathcal{U}_{L'}(l')$ 
  are open neighborhoods of $l$ and $l'$, then 
  \begin{align*}
    C \cap p^{-1}(U)\times q^{-1}(V) \neq \emptyset.
  \end{align*}
  Thus, there is an $\alpha_0 \in A$ such that for all $\alpha \geq \alpha_0$
  \begin{align*}
    \Gr(\theta_\alpha) \cap p^{-1}(U)\times q^{-1}(V) \neq \emptyset.
  \end{align*}
  Since $\theta_\alpha\in\uC_p^q(K, Y)$,
  it follows that $\Gr(\theta_\alpha) \subset  p^{-1}(U)\times q^{-1}(V)$ for $\alpha \geq \alpha_0$
  and hence that 
  	\begin{align*}
  		C \subset \overline{p^{-1}(U)}\times \overline{q^{-1}(V)}.
  	\end{align*}
  Since $U$ and $V$ were arbitrary, $C \subset K_l \times Y_{l'}$.
  
  Since $p$ is open, it follows that for every $x\in K_l$ there is a 
  $y\in Y_{l'}$ such that $(x, y) \in C$: Use \cref{rem:quotientmapopenchar} and 
  the compactness of $\im(\mathcal{F})$ to find a 
  subnet $(\Gr(\theta_\beta))_{\beta\in B}$ and a net $(x_\beta)_{\beta \in B}$ such 
  that $(x_\beta)_{\beta \in B}$ converges to $x$, $p(x_\beta) = s(\theta_\beta)$ for 
  every $\beta \in B$, and $(\theta_\beta(x_\beta))_{\beta \in B}$ converges to some 
  $y\in Y$. Since $(\Gr(\theta_\beta))_{\beta\in B}$ converges to $C$ with respect to the 
  Vietoris topology, this then shows that $(x, y) \in C$.
  In order to see that $C$ is, in fact, the graph of a function $\theta\colon K_l\to Y_{l'}$, 
  assume that 
  $(x, y), (x, y') \in C$. Then there are nets $(x_\alpha, \theta_\alpha(x_\alpha))_{\alpha\in A}$,
  $(x_\alpha', \theta_\alpha(x_\alpha'))_{\alpha\in A}$ converging to $(x, y)$ and $(x, y')$. It
  then follows from the equicontinuity of $\euF$ that the nets $(\theta_\alpha(x_\alpha))_{\alpha \in A}$
  and $(\theta_\alpha(x_\alpha'))_{\alpha \in A}$ have the same limits. This shows that $y = y'$, i.e., 
  there is a function $\theta\colon K_l \to Y_{l'}$ with $C = \Gr(\theta)$. Since $K_l$ is compact
  and $Y_{l'}$ is Hausdorff, the closed graph theorem shows that $\theta$ is continuous, i.e.,
  $\theta\in \uC_p^q(K,Y)$. Hence, $\euF$ is precompact.
  
  For the converse implication, we may assume $\euF$ to be compact. Using the characterization of
  convergent nets in the compact-open topology from \cref{charconv}, it 
  is then easy to see that $\im(\euF)$ is compact. If $\euF$ were not 
  equicontinuous, we would find a net $((\theta_\alpha,x_\alpha,x_\alpha'))_{\alpha \in A}$ in 
  $\euF \times_L K \times_L K$ and a $U \in \mathcal{U}_Y$ such that 
  $\lim_\alpha x_\alpha = \lim_\alpha x_\alpha'$ and 
  $(\theta_\alpha(x_\alpha),\theta_\alpha(x_\alpha')) \notin U$ 
  for every $\alpha\in A$ which clearly contradicts the compactness of $\euF$.
  Thus, $\euF$ is equicontinuous.
\end{proof}

\begin{corollary}\label{compactnessvsequicontinuity}
  For a groupoid action $(K,q,\euG)$ the following assertions are equivalent.
	\begin{enumerate}[(a)]
		\item\label{item:comp1} $(K,q,\euG)$ is equicontinuous.
		\item\label{item:comp2} $\euS(K,q,\euG) \subset \uC_q^q(K, K)$ is precompact.
		    \item\label{item:comp3} $\{f|_{K_{r(\pzg)}} \circ \varphi_\pzg \mid \pzg \in \euG\} \subset \uC_q(K)$ is 
    equicontinuous for all $f\in\uC(K)$ from one/every subset $M$ of $\uC(K)$
    that generates $\uC(K)$ as a $\uC^*$-algebra.
    \item\label{item:comp4} $\{f|_{K_{r(\pzg)}} \circ \varphi_\pzg \mid \pzg \in \euG\} \subset \uC_q(K)$ is 
    relatively compact for all $f\in\uC(K)$ from one/every subset $M$ of $\uC(K)$
    that generates $\uC(K)$ as a $\uC^*$-algebra.
	\end{enumerate}	  
\end{corollary}
\begin{proof}
  Given \cref{arzelaascoli}, the equivalence of \ref{item:comp1} and 
  \ref{item:comp2} is hard not to prove. Similarly, \ref{item:comp3}
  and \ref{item:comp4} are equivalent. So suppose $M \subset \uC(K)$ 
  is as in \ref{item:comp3}. Straightforward arguments show that the 
  property 
  \begin{align*}
   \{f|_{K_{r(\pzg)}} \circ \varphi_\pzg \mid \pzg \in \euG\} \subset \uC_q(K) \text{ is equicontinuous}
  \end{align*}
  is preserved under taking finite linear combinations,
  products, and conjugates of functions in $\uC(K)$. Thus, we may 
  assume that $M$ is dense in $\uC(K)$. Now, to verify the equicontinuity 
  of $(K, q, \euG)$, let $V\in\mathcal{U}_K$ be a given entourage. Since 
  the functions in $\uC(K)$ generate the uniformity on $K$ and $M$ is 
  dense, we can find an $\epsilon > 0$ and an $f\in M$ such that 
  $U_{f, \epsilon} \subseteq V$ where 
  \begin{align*}
    U_{f, \epsilon} = \{ (x, y) \in K\times K \mid |f(x) - f(y)| < \epsilon \}.
  \end{align*}
  By assumption,
  \begin{align*}
    \{f|_{K_{r(\pzg)}} \circ \varphi_\pzg \mid \pzg \in \euG\}
  \end{align*}
  is equicontinuous and so we may find an entourage 
  $U\in\mathcal{U}_K$ such that for all $\pzg\in \euG$ and all 
  $(x, y)\in U$ with $q(x) = q(y) = s(\pzg)$ one has 
  \begin{align*}
    |f(\pzg x) - f(\pzg y)| < \epsilon. 
  \end{align*}
  In other words, $\euG$ maps $K\times_q K \cap U$ into $U_{f, \epsilon} \subset V$,
  so $\euG$ is equicontinuous and \ref{item:comp3} implies \ref{item:comp1}. 
  The converse implication is again easy to verify.
\end{proof}

In particular, if $\euE_\uu(K,q,\euG)$ is compact, $(K,q,\euG)$ is necessarily equicontinuous. The following
example shows that the converse is generally not true because the 
inclusion $\overline{\euS(
K,q,\euG)} \subset \euE_\uu(K,q,\euG)$ is generally strict, as noted in \cref{rem:closurenotsgrpd} and 
\cref{ex:twoxsquared}.

\begin{example}\label{ex:equinoniso}
  Let $L_0 \defeq [0, \infty)$ and 
  \begin{align*}
    \psi_0\colon L_0\to L_0, \quad 
    \psi_0(x) \defeq \floor*{x} + \left(x-\floor*{x}\right)^2
  \end{align*}
  as well as $K_0 \defeq L_0\times \Z_2$ and
  \begin{align*}
    \phi_0\colon K_0 \to K_0, \quad \phi_0(x, g) \defeq (\psi_0(x), g + 1).
  \end{align*}
  Then $q_0\colon K_0 \to L_0$, $(x, h) \mapsto x$ is continuous and 
  intertwines $\phi_0$ and $\psi_0$. Since $\psi_0$, $\phi_0$, and $q$ are proper, 
  they extend canonically to the one-point compactifications $K \defeq K_0 \cup \{\infty_{K_0}\}$ and 
  $L \defeq L_0 \cup \{\infty_{L_0}\}$ of $K_0$ and 
  $L_0$ and thereby yield an extension $q\colon (K, \phi) \to (L, \psi)$ of topological 
  dynamical systems. It is easy to see that $\overline{\euS(q)}$ is compact since
  \begin{align*}
    \overline{\euS(q)} \subset \left\{\theta_\infty\right\} \cup 
      \bigcup_{n\in\N_0}\left\{\theta_{x, y, g} \mmid x,y \in [n, n+1], g\in \Z_2\right\}
  \end{align*}
  where for $x, y\in L$ and $g\in \Z_2$, we define $\theta_{x, y, g}$ and $\theta_x$ as
  \begin{alignat*}{2}
    \theta_{x, y, g}&\colon K_x \to K_y, \quad &(x, h) &\mapsto (y, g+h), \\
    \theta_{x}&\colon K_x\to \left\{\infty_{K_0}\right\}, \quad &(x, h) &\mapsto \infty_{K_0}.
  \end{alignat*}
  However, 
  \begin{align*}
    \euE_\uu(q) = \left\{\theta_{x, y, g} \mmid x, y\in L_0, g\in \Z_2\right\}
    \cup \{\theta_{x} \mid x\in L\}
  \end{align*}
  and since $\theta_x$ is not invertible for $x \neq \infty_{L_0}$, $\euE_\uu(q)$ is neither 
  a groupoid nor compact (use \cref{compactimpliesgroupoid}).
\end{example}

Thus, in contrast to the case of group actions, in order to 
characterize the compactness of $\euE_\uu(K,q,\euG)$, a more restrictive property than 
equicontinuity is needed. The following proposition shows that pseudoisometry is a sufficient 
condition for the enveloping semigroupoid to be a compact groupoid.

\begin{proposition}\label{prop:pseudo_impl_groupoid}
  Let $(K,q,\euG)$ be a pseudoisometric groupoid action.
  Then $\euE_\uu(K,q,\euG)$ is a compact groupoid.
\end{proposition}
\begin{proof}
  Pick a set $P$ as in \cref{def:struct_ext} \ref{def:struct_ext_item:pseudoiso} and consider the set
  \begin{align*}
    \mathrm{I}(P) \defeq \left\{\theta \in \uC_q^q(K, K) \mmid 
      \begin{matrix}
        \theta\colon K_{s(\theta)} \to K_{r(\theta)} \text{ is bijective and for all } p \in P, \\
        x,y \in K_{s(\theta)} \text{ one has } p(\theta(x),\theta(y)) = p(x,y) 
      \end{matrix}
    \right\}.
  \end{align*}
  By \cref{arzelaascoli}, $\mathrm{I}(P)$ is a compact (semi)groupoid containing 
  $\euS(K,q,\euG)$ and therefore 
  $\euE_\uu(K,q,\euG) \subset \mathrm{I}(P)$ is itself a compact semigroupoid. It follows 
  from \cref{compactimpliesgroupoid} above that it is in fact a groupoid.
\end{proof}

The following proposition shows that if $\euE_\uu(K, q, \euG)$ is transitive, then we can actually characterize pseudoisometric exensions via the compactness of the uniform enveloping semigroupoid.

\begin{proposition}\label{prop:transcompacttopseudo}
  Let $(K, q, \euG)$ be a groupoid action such that $\euE_\uu(K, q, \euG)$ is a compact transitive 
  groupoid. Then $(K, q, \euG)$ is pseudoisometric.
\end{proposition}
\begin{proof}
  Let 
  $P$ be a family of pseudometrics 
  generating the topology of $K$. Then for $p\in P$, 
  define 
  \begin{align*}
    p' \colon K \times_{\euG^{(0)}} K \to [0,\infty), \quad 
   (x, y) \mapsto \max_{\substack{\theta \in \euE_\uu(K, q, \euG)\\ s(\theta) = q(x)}}
    p(\theta(x),\theta(y)).
  \end{align*}
  Then the family $P' \defeq \{p' \mid p\in P\}$ generates the topology of $K_\pzu$ for each $\pzu\in \euG^{(0)}$
  since $\euE_\uu(K,q,\euG)_{\pzu}$ is compact.
  Moreover, since the range and source map of a compact transitive groupoid are open 
  by \cref{lem:transitive_open} below, each $p'$ is continuous and one readily verifies
  the invariance of the $p'$.
\end{proof}

\begin{proposition}\label{lem:transitive_open}
  Let $\euG$ be a compact transitive groupoid. Then $(s,r)$, $s$, and $r$ are open and so is the 
  restriction $p$ of $s$ and $r$ to $\Iso(\euG)$. 
\end{proposition}
\begin{proof}
  We start with the restrictions to $\Iso(\euG)$:
  Pick $\pzg \in \mathrm{Iso}(\euG)$ and set $\pzu\defeq p(\pzg) \in \euG^{(0)}$. Moreover, let
  $(\pzu_\alpha)_{\alpha \in A}$ be a net in $\euG^{(0)}$ converging to $\pzu$. For each $\alpha \in A$ we 
  there is an $\pzh_\alpha \in \euG_{\pzu}^{\pzu_\alpha}$ and by passing to a subnet, we may assume that 
  $\lim_\alpha \pzh_\alpha = \pzh \in \euG_u^u$. But then $\pzg = \lim_\alpha \pzh_\alpha (\pzh^{-1}\pzg\pzh) \pzh_\alpha^{-1}$
  and so we have found a net $(\pzg_\alpha)_{\alpha\in A}$ in $\Iso(\euG)$ that converges to 
  $\pzg$ and satisfies $r(\pzg_\alpha) = \pzu_\alpha$ for every $\alpha \in A$. Thus, $r$ is open.
  
  To show that $(s, r)$, $s$, and $r$ are open, it suffices to show that $(s, r)$ is open, so let $\pzg\in \euG$
  and $(\pzu_\alpha, \pzv_\alpha)_{\alpha\in A}$ be a net in $\euG^{(0)}\times\euG^{(0)}$ converging to 
  $(u,v) = (s(\pzg), r(\pzg))$. Since $\euG$ is transitive, there is a net $(\pzh_\alpha)_{\alpha\in A}$ in 
  $\euG$ with $s(\pzh_\alpha) = \pzu_\alpha$ and $r(\pzh_\alpha) = \pzv_\alpha$ for each $\alpha \in A$.
  By compactness of $\euG$, we may assume that $(\pzh_\alpha)_{\alpha\in A}$ converges to 
  some element $\pzh \in \euG$ in with $s(\pzh) = s(\pzg)$ and $r(\pzh) = s(\pzg)$. Set $\gamma \defeq \pzg\pzh^{-1} \in \Iso(\euG)_{r(g)}$ 
  and, using the openness result for the isotropy bundle, find, after possibly passing to a subnet, a 
  net $(\gamma_\alpha)_{\alpha\in A}$ in $\Iso(\euG)$
  with $p(\gamma_\alpha) = v_\alpha$ for each $\alpha \in A$. Then the net 
  $(\gamma_\alpha \pzh_\alpha)_{\alpha \in A}$ converges to $\pzg$ and satisfies 
  	\begin{align*}
  		(s(\gamma_\alpha\pzh_\alpha), r(\gamma_\alpha\pzh_\alpha))  = (s(\pzh_\alpha), r(\pzh_\alpha)) = (\pzu_\alpha, \pzv_\alpha)
  	\end{align*}
  for each $\alpha \in A$. Hence, $(s, r)$ is open.
\end{proof}

\subsection{Characterizing transitivity.} \cref{prop:transcompacttopseudo} is unsatisfying in 
that it is not yet clear when $\euE_\uu(K, q, \euG)$ is a transitive groupoid. Therefore, we 
show in this subsection that the transitivity of $\euE_\uu(K, q, \euG)$ can be characterized 
purely in terms of $\euG$. To this end, recall from \cref{ex:unitspaceaction} that a groupoid 
is transitive if and only if the action on its unit space is transitive. This allows to reduce
the question when $\euE_\uu(K, q, \euG)$ is transitive to a question purely about $\euG$
and its action $(\euG^{(0)}, \id_{\euG^{(0)}}, \euG)$ on its unit space. The following lemma 
and \cref{cor:transviaunitspace} show that we thus only need to consider the question when 
$\euE_\uu(\euG^{(0)}, \id_{\euG^{(0)}}, \euG)$ is transitive.

\begin{lemma}\label{lem:factorgroupoid}
  Let $p \colon (K_1,q_1,\euG) \to (K_2,q_2,\euG)$ be an extension of groupoid actions. 
  If $\euE_\uu(K_1,q_1,\euG)$ is compact, then
  \begin{align*}
    \Phi_p \colon \euE_\uu(K_1,q_1,\euG) \to \euE_\uu(K_2,q_2,\euG), \quad \vartheta \to \Phi_p(\vartheta)
  \end{align*}     
  is a factor map of topological groupoids where
  \begin{align*}
    \Phi_p(\vartheta)\colon (K_2)_{s(\vartheta)} \to (K_2)_{r(\vartheta)}, \quad p(x) \mapsto p(\vartheta(x))   
  \end{align*}
  for $\vartheta \in  \EuScript{E}_{\mathrm{u}}(q_1)$. Moreover, 
  \begin{align*}
    \Phi_p^{(0)} \colon \euE_\uu(K_1,q_1,\euG)^{(0)} \to \euE_\uu(K_2,q_2,\euG)^{(0)}
  \end{align*}
  is bijective.
\end{lemma}
\begin{proof}
  We first check that $\Phi_p$ is well-defined. Let $\euS$ be the set of all elements 
  $\vartheta \in \EuScript{E}_{\mathrm{u}}(K_1,q_1,\euG)$ with the following property: If 
  $x,y \in (K_1)_{s(\vartheta)}$ with $p(x) = p(y)$, then $p(\vartheta(x)) = p(\vartheta(y))$. 
  Then $\euS$ is a semigroupoid containing $\euS(K_1,q_1,\euG)$ and we show that is is closed in 
  $\EuScript{E}_{\mathrm{u}}(K_1,q_1,\euG)$. Let $(\vartheta_\alpha)_{\alpha \in A}$ be a net in 
  $\euS$ converging to $\vartheta \in \EuScript{E}_{\mathrm{u}}(K_1,q_1,\euG)$ and 
  $x,y \in (K_1)_{s(\vartheta)}$ with $p(x) = p(y)$. Since $p$ and $q_2$ are open, we 
  find, by passing to a subnet, a $((x_\alpha,y_\alpha))_{\alpha \in A}$ in $K_1 \times K_1$ such 
  that $x = \lim_\alpha x_\alpha$, $y = \lim_\alpha y_\alpha$ and $p(x_\alpha) = p(y_\alpha)$ as well 
  as $q_1(x_\alpha) = q_1(y_\alpha) = s(\vartheta_\alpha)$ for every $\alpha \in A$. But then 
  $p(\vartheta_\alpha(x_\alpha)) = p(\vartheta_\alpha(y_\alpha))$ for every $\alpha \in A$ and therefore
  \begin{align*}
    p(\vartheta(x)) = \lim_\alpha p(\vartheta_\alpha(x_\alpha)) = \lim_\alpha  p(\vartheta_\alpha(y_\alpha)) = p(\vartheta(y)).
  \end{align*}
  Thus, $\euS$ is closed and therefore $\euS = \EuScript{E}_{\mathrm{u}}(K_1,q_1,\euG)$. It is now clear, that
    \begin{align*}
      \Phi_p\colon \EuScript{E}_{\mathrm{u}}(K_1,q_1,\euG) \to \mathrm{C}_{q_2}^{q_2}(K_2,K_2), \quad \vartheta \mapsto \Phi_p(\vartheta)
    \end{align*}
  is a well-defined morphism of semigroupoids and a moment's thought revals that it is continuous. 
  Since $\EuScript{E}_{\mathrm{u}}(K_1,q_1,\euG)$ is compact, we obtain that its image is a closed 
  subsemigroupoid of $\mathrm{C}_{q_2}^{q_2}(K_2,K_2)$ containing $\euS(q_2)$ and therefore containing 
  $\EuScript{E}_{\mathrm{u}}(K_2,q_2,\euG)$. On the other hand, 
  $\Phi_p^{-1}(\EuScript{E}_{\mathrm{u}}(K_2,q_2,\euG))$ is a closed subsemigroupoid of 
  $\EuScript{E}_{\mathrm{u}}(K_1,q_1,\euG)$ containing $\euS(K_1,q_1,\euG)$ showing that the image 
  of $\Phi_p$ is precisely $\EuScript{E}_{\mathrm{u}}(K_2,q_2,\euG)$.
  Moreover, $\Phi_p^{(0)}$ is easily shown to be bijective since $p$ is an 
  extension of groupoid actions, i.e., $q_1 = q_2\circ p$.
\end{proof}

\begin{corollary}\label{cor:transviaunitspace}
  Let $(K, q, \euG)$ be a groupoid action such that its enveloping semigroupoid $\euE_\uu(K, q, \euG)$ 
  is compact.
  Then $\euE_\uu(K, q, \euG)$ is transitive if and only if $\euE_\uu(\euG^{(0)}, \id_{\euG^{(0)}}, \euG)$ is 
  transitive.
\end{corollary}
\begin{proof}
  Consider the extension $q\colon (K, q, \euG) \to (\euG^{(0)}, \id_{\euG^{(0)}}, \euG)$ of 
  groupoid actions. Then \cref{lem:factorgroupoid} shows that there is a surjective groupoid 
  morphism from $\euE_\uu(K, q, \euG)$ to $\euE_\uu(\euG^{(0)}, \id_{\euG^{(0)}}, \euG)$ which
  is bijective on the level of unit spaces. Thus, $\euE_\uu(K, q, \euG)$ is transitive if 
  and only if $\euE_\uu(\euG^{(0)}, \id_{\euG^{(0)}}, \euG)$ is.
\end{proof}

The question that now remains is: When is $\euE_\uu(\euG^{(0)}, \id_{\euG^{(0)}}, \euG)$ transitive?
To understand this, first recall from \cref{ex:envelopingrelation} that the groupoid
$\euE_\uu(\euG^{(0)}, \id_{\euG^{(0)}}, \euG)$ is isomorphic to $\euE_\uu(R_\euG)$ where 
$R_\euG$ is the orbit relation on $\euG^{(0)}$. Therefore, we need to understand when the 
equivalence relation $\euE_\uu(R_\euG)$ is transitive which, as noted in \cref{ex:pairgroupoid}, amounts to 
understanding when $\euE_\uu(R_{\euG}) = L\times L$. We now consider the following illustrating examples.

\begin{example}\label{ex:transitivity} Let $(L, G)$ be a topological dynamical system and let 
  $G \ltimes L$ be the action groupoid of $(L, G)$. In this case, 
  $R_{G\ltimes L} = \{ (x, y) \mid \exists g\in G\colon gx = y\}$ is the regular orbit relation 
  on $L$.
  \begin{enumerate}
    \item If every orbit in $(L, G)$ equals $L$, then $R_{G\times L} = L\times L$, $G \ltimes L$ 
    is transitive, and so is $\euE_\uu(R_{G\times L})$. However, this case almost never occurs in 
    topological dynamics.
    \item If $(L, G)$ has a dense orbit, then $R_{G\ltimes L}$ is dense in $L\times L$ 
    and so $\euE_\uu(R_{G\ltimes L}) = L\times L$, so $\euE_\uu(R_{G\ltimes L})$ is transitive.
    \item Even if $(L, G)$ does not have any transitive point, $\euE_\uu(L, \id_L, G \ltimes L)$ 
    may still be transitive. To see this, revisit the system $(L, \psi)$ considered in \cref{ex:twoxsquared}:
    It follows either by direct computation or by observing the transitivity of 
    $\euE_\uu(K, q, \Z \ltimes L)$ that $\euE_\uu(R_{G\ltimes L})$ is transitive. However,
    $(L, \Z)$ itself is not transitive.
    \item Consider the system $(L, \psi)$ on $L \defeq [0,1]\times\Z_2$ given by the map
    \begin{align*}
      \psi\colon [0,1]\times\Z_2 \to [0,1]\times\Z_2, \quad (x, g) \mapsto (x^2, g).
    \end{align*}
    Then 
    \begin{align*}
      \euE_\uu(R_{G\ltimes L}) 
      = \left\{ ((x, g), (y, h)) \in [0,1]\times\Z_2 \mmid g = h \right\} \subsetneq L\times L.
    \end{align*}
    In particular, $\euE_\uu(R_{G\ltimes L})$ is not transitive.
  \end{enumerate}
\end{example}

\begin{remark}\label{rem:topologicallyergodic}
  In light of \cref{ex:transitivity}, it is apparent that the orbit structure of the action of 
  a groupoid $\euG$ on its unit space $\euG^{(0)}$ plays an essential role for the transitivity
  of the enveloping groupoid $\euE_\uu(\euG^{(0)}, \id_{\euG^{(0)}}, \euG)$ and the 
  equivalence relation $\euE_\uu(R_\euG)$. For topological dynamical systems $(L, \psi)$, 
  it has been characterized when $\euE_\uu(R_{G\ltimes L})$, the smallest closed equivalence 
  relation containing the orbit relation, is all of $L\times L$: It is equivalent 
  each of the following assertions.
  \begin{enumerate}[(i)]
    \item The fixed space $\fix(T_\psi)$ of the Koopman operator $T_\psi\colon\uC(L) \to \uC(L)$
    is one-dimensional.
    \item The maximal trivial factor of $(L, \psi)$ is a point.
  \end{enumerate}
  See \cite{Kuester2019} and \cite[Section 1]{Edeko2020} for more information. In analogy with 
  ergodic measure-preserving systems, we call such systems $(L, G)$ \emph{topologically ergodic}. 
  We now extend this characterization to groupoid actions.
\end{remark}


\begin{definition}
  A factor $(M, t, \euH)$ of $(K, q, \euG)$ with factor map $(p, \Phi)$ is called a \emph{trivial factor} if the acting groupoid $\euH$ is trivial in the sense of \cref{ex:trivialgrpd}. It is a
  \emph{maximal trivial factor}, if for any factor map $(\tilde{p}, \tilde{\Phi})\colon (K, q, \euG) \to (\tilde{M}, \tilde{t}, \tilde{\euH})$ onto 
  another trivial factor there is a unique factor map $(m,\Theta)$ such that the following diagram commutes. 
  \begin{align*}
    \xymatrix{
      (K, q, \euG) \ar[rr]^-{(p,\Phi)} \ar[rd]_-{(\tilde{p},\tilde{\Phi})} && (M, t, H) \ar@{-->}[ld]^-{\exists!\, (m,\Theta)} \\
      & (\tilde{M}, \tilde{t}, \tilde{\euH})  & 
    }
  \end{align*}
  We call $(K, q, \euG)$ \emph{topologically ergodic} 
  if every trivial factor is a point and say that a groupoid $\euG$ is \emph{topologically ergodic} if its 
  action $(\euG^{(0)}, \id_{\euG^{(0)}}, \euG)$ on its unit space is topologically ergodic.
\end{definition}

\begin{lemma}\label{orbitsellisgroupoid1}
  Let $(K, q, \euG, \phi)$ be a groupoid action. Then the folowing assertions hold. 
    \begin{enumerate}[(i)]
      \item Maximal trivial factors are unique up to isomorphy.
      \item If 
  \begin{align*}
     R_\phi = \left\{ (x,y)\in K\times K \mid y \in \euG x\right\}
  \end{align*}
  denotes the orbit relation, then the space $K/\euE_\uu(R_\phi)$ defines a maximal trivial factor of $(K, q, \euG)$.
    \end{enumerate} 
\end{lemma}
\begin{proof}
  Two trivial factors $(M_1, t_1, \euH_1)$ and $(M_2, t_2, \euH_2)$ are isomorphic if and only if 
  the associated equivalence relations on $K$ agree. 
  By construction, $\euE_\uu(R_\phi)$ is the smallest closed equivalence relation on 
  $K$ that contains the equivalence relation $R_\phi$.
  Thus, $K/\euE_\uu(R_\phi)$ is a maximal trivial factor of $(K, q, \euG)$ and every other maximal 
  trivial factor is isomorphic to it.
\end{proof}

In view of \cref{orbitsellisgroupoid1}, we may from now on speak of \emph{the} maximal trivial factor 
$\fix(K, q, \euG)$ of a groupoid action $(K, q, \euG)$. Summarizing our obsevations, 
we obtain the following characterization.

\begin{theorem}\label{chartrans}
  Let $\euG$ be a topological groupoid with compact unit space. Then the 
  following assertions are equivalent.
  \begin{enumerate}[(i)]
    \item $\euG$ is topologically ergodic.
    \item The equivalence relation $\euE_\uu(R_\euG)$ equals $\euG^{(0)}\times\euG^{(0)}$.
    \item The enveloping groupoid $\euE_\uu(\euG^{(0)}, \id_{\euG^{(0)}}, \euG)$ is transitive.
    \item The enveloping groupoid $\euE_\uu(K, q, \euG)$ is transitive for every action 
    $(K, q, \euG)$ of $\euG$ such that $\euE_\uu(K, q, \euG)$ is compact.
  \end{enumerate}
\end{theorem}

Since we are ultimately interested in groupoid actions that arise from extensions of topological 
dynamical systems as in \cref{ex:extensiongrpdaction}, the following corollary provides a 
simple criterion for topological ergodicity.

\begin{corollary}
  Let $(K, G)$ be a topological dynamical system. Then the action groupoid $G \ltimes K$ is 
  topologically ergodic if and only if the system $(K, G)$ is.
\end{corollary}

We are now ready to state the main result of this section.

\begin{theorem}\label{thm:pseudoisochar}
  Let $(K,q,\euG)$ be a groupoid action by a topologically ergodic groupoid $\euG$. Then the 
  following assertions are equivalent.
  \begin{enumerate}
    \item\label{thm:pseudoisochar_pseudoiso} $(K, q, \euG)$ is pseudoisometric.
    \item\label{thm:pseudoisochar_groupoid} $\euE_\uu(K, q, \euG)$ is a compact groupoid.
  \end{enumerate}
\end{theorem}
\begin{proof}
  The implication \ref{thm:pseudoisochar_pseudoiso} $\implies$ \ref{thm:pseudoisochar_groupoid}
  was established more generally in \cref{prop:pseudo_impl_groupoid}. The converse implication
  follows from \cref{prop:transcompacttopseudo} since if $\euG$ is topologically ergodic,
  $\euE_\uu(K, q, \euG)$ is a compact transitive groupoid by \cref{chartrans}.
\end{proof}

\begin{remark}\label{remark:metrizable}
	Let $(K,q,\euG)$ be a groupoid action by a topologically ergodic groupoid $\euG$. If $K$ is metrizable, 
	then the proof above reveals that $(K,q,\euG)$ is isometric if and only if it is pseudoisometric.
\end{remark}

\subsection{Maximal trivial factors and Koopman representations.}

As noted in \cref{rem:topologicallyergodic}, topological ergodicity of a dynamical system 
$(L, \psi)$ can be characterized in terms of its Koopman operator, allowing for a very 
convenient criterion. We extend this characterization to groupoids.

\begin{definition}\label{def:koopmanrep}
  Let $(K,q,\euG,\varphi)$ be a groupoid action. The map
  \begin{align*}
    T_\varphi \colon \euG \to 
    \bigcup_{\pzu,\pzv \in \euG^{(0)}} \mathscr{L}(\mathrm{C}(K_\pzu),\mathrm{C}(K_\pzv)), \quad 
    \pzg \mapsto T_{\pzg}
  \end{align*}
  with $T_\pzg f \defeq f \circ \varphi_{\pzg^{-1}}$ for $f \in \uC(K_{\mathrm{s}(\pzg)})$ is called 
  the \emph{Koopman representation of $(K,q,\euG,\varphi)$}. Moreover, the set
  \begin{align*}
    \fix(T_\varphi)
    \defeq \left\{f \in \mathrm{C}(K) \mmid \forall \pzg \in \euG \colon T_\pzg \left(f|_{K_{s(\pzg)}}\right)
    = f|_{K_{r(\pzg)}}\right\}
  \end{align*}
  is called its \emph{fixed space}. If $\euG$ is a topological groupoid with compact unit space, we 
  write $T_\euG$ for the Koopman representation of $(\euG,\mathrm{id}_{\euG^{(0)}},\euG^{(0)})$ and 
  call this the \emph{Koopman representation of $\euG$}.
\end{definition}

\begin{remark}\label{rem:fix} 
  If $(K,q,\euG,\varphi)$ is a groupoid action, we can recover its fixed space from the action 
  groupoid $\euG \ltimes K$ (see \cref{def:actiongrpd2}). Concretely, we obtain, under the usual 
  identification of the unit space, the identity $\fix(T_\varphi) = \fix(T_{\euG \ltimes K})$.
\end{remark}

The fixed space of the Koopman representation is always a unital commutative 
C*-algebra and therefore isomorphic to $\mathrm{C}(X)$ where $X$ is its (compact) Gelfand space. Using this observation we obtain the following result characterizing the maximal trivial factor of a groupoid action. 

\begin{proposition}
	Let $(K,q,\euG,\varphi)$ be a groupoid actions. Then the Gelfand space of the fixed space $\fix(T_\varphi)$ defines a maximal trivial factor of $(K,q,\euG,\varphi)$.
\end{proposition}
\begin{proof}
	The Gelfand space of $\fix(T_\varphi)$ is homeomophic to the compact space $M = K/{R_{\mathrm{fix}}}$ with
	\begin{align*}
  R_{\mathrm{fix}}
  \defeq \left\{(x, y)\in K \times K \mmid \forall f\in \fix(T_\varphi)\colon f(x) = f(y)\right\}.
  \end{align*}
	Clearly, $R_{\mathrm{fix}}$ is a closed and invariant equivalence relation containg the orbit relation $R_\varphi$. We therefore immediately obtain that $\euE_\uu(R_\phi) \subseteq R_{\mathrm{fix}}$. On the other hand, if $R$ is any closed invariant equivalence relation and $\pi_R \colon K \rightarrow K/R$ the induced map, then $(x,y) \in R$ if and only if $f(\pi_R(x)) = f(\pi_R(y))$ for every $f \in \mathrm{C}(K/R)$. However, $T_{\pi_R}f = f \circ \pi_R \in \fix(T_\varphi)$ for every $f \in \mathrm{C}(K/R)$. This shows $R_{\mathrm{fix}} \subseteq \euE_\uu(R_\phi)$ and consequently $R_{\mathrm{fix}} = \euE_\uu(R_\phi)$. The claim now follows from \cref{orbitsellisgroupoid1}.
\end{proof}
\begin{corollary}
	A groupoid action $(K,q,\euG,\varphi)$ is topologically ergodic if and only if $\fix(T_\varphi)$ is one-dimensional.
\end{corollary}
\begin{corollary}
	A groupoid $\euG$ with compact unit space $\euG^{(0)}$ is topologically ergodic if and only if $\fix(T_\euG)$ is one-dimensional.
\end{corollary}

\section{Representations of compact transitive groupoids}
\label{sec:rep}

In this section we study the representation theory of compact transitive groupoids and apply 
it to the uniform enveloping (semi)groupoids of pseudoisometric groupoid actions. We start by 
recalling the following consequence of the Peter--Weyl theorem for representations of compact 
groups (see \cite[Theorem 15.14]{EFHN2015}).

\begin{theorem}\label{thm:PW_appl}
  Let $T \colon G \to \mathscr{L}(E)$ be a strongly continuous representation of a compact 
  group $G$ on a Banach space $E$. Then the following assertions hold.
  \begin{enumerate}[(i)]
  	\item The union of all finite-dimensional invariant subspaces
  of $E$ is dense in $E$.
 	\item If $G$ is abelian, then the union of all one-dimensional invariant subspaces
  of $E$ is dense in $E$.
  \end{enumerate}
\end{theorem}

We prove a generalization of this result to representations 
of compact transitive groupoids
in \cref{pwtransitive}. We then apply this generalization to prove 
\cref{mainthm}, the main result of this section 
that characterizes pseudoisometric groupoid actions. To perform this
generalization, we need to start by replacing Banach spaces 
by Banach bundles (see, e.g., \cite[Definition 1.1]{DuGi1983} or 
\cite[Section 1 and Theorem 3.2]{Gierz1982}).

\begin{definition}\label{def:banachbundle}
	Let $L$ be a compact space. A \emph{Banach bundle} over $L$ is a topological space $E$ together 
	with a continuous open surjection $p \colon E \to L$ with the following properties.
  \begin{enumerate}[(i)]
    \item Every fiber $E_l$ is a Banach space.
    \item The mappings
      \begin{alignat*}{3}
        + &\colon E \times_L E \to E, \quad &(e,f) &\mapsto e+f\\
        \cdot &\colon \C \times E \to E, \quad &(\lambda,e) &\mapsto \lambda e
      \end{alignat*}
      are continuous.
    \item The norm mapping
      \begin{align*}
        \|\cdot\|\colon E \to [0,\infty), \quad e \mapsto \|e\|
      \end{align*}
      is upper semicontinuous.
    \item For each $l \in L$ the sets
      \begin{align*}
        \{e \in E\mid p(e) \in U, \|e\| < \epsilon\}
      \end{align*}
      for neighborhoods $U \subset L$ of $l$ and $\epsilon > 0$ define a neighborhood base of 
      $0_l \in E_l$.
  \end{enumerate}
  
	A Banach bundle $E$ is
  \begin{itemize}
    \item \emph{continuous} if the norm mapping $\|\cdot\|$ of (iii) is continuous,
    \item \emph{of constant dimension $n$} for some $n \in \N_0$ if $\dim (E_l) = n$ for every $l \in L$. 
    \item \emph{of constant finite dimension} if it is of constant dimension $n$ for some $n \in \N_0$.
    \item \emph{locally trivial} if for each $l \in L$ there are a compact neighborhood 
    $W$ of $l$, $n\in \N_0$ and a homeomorphism $\Phi\colon p^{-1}(W) \to W \times \C^n$ with the following properties.
    \begin{enumerate}[-]
      \item The diagram 
        \[
          \xymatrix{p^{-1}(W) \ar[rr]^\Phi \ar[rd]_p & &W \times \C^n \ar[ld]^{\mathrm{pr}_1} \\
          & W &
          }
        \]
        commutes where $\mathrm{pr}_1\colon W \times \C^n \to W$ is the prorjection onto the first component.
      \item $\Phi|_{E_l} \colon E_l \to \{l\} \times \C^n$ is an isomorphism of vector spaces for every $l \in W$.
      \item There are constants $c_1, c_2 > 0$ such that 
        \begin{align*}
          c_1 \cdot \|e\| \leq \left\|\mathrm{pr}_2(\Phi(e))\right\| \leq c_2 \cdot \|e\|
        \end{align*}
        for every $e \in p^{-1}(W)$ where $\mathrm{pr}_2 \colon W \times \C^n \to \C^n$ is the projection onto the second component.
    \end{enumerate}
  \end{itemize}
	Moreover, we write
		\begin{align*}
			\Gamma(E) \defeq \{\sigma \in \uC(L,E)\mid p \circ \sigma = \id_L\} 
		\end{align*}
	for the \emph{space of continuous sections of $E$}.
\end{definition}

\begin{remark}\label{remarksBB}
	\begin{enumerate}[(i)]
		\item \label{remarksBB_module} If $E$ is a Banach bundle over a compact space $L$, then $\Gamma(E)$ is canonically a 
		module over $\uC(L)$ and a Banach space with the norm defined by 
		$\|\sigma\|\defeq \sup_{l \in L}\|\sigma(l)\|_{E_l}$ for $\sigma \in \Gamma(E)$. Moreover, 
		$\|f\sigma\| \leq \|f\|\cdot \|\sigma\|$ for all $f \in \uC(L)$ and $\sigma \in \Gamma(E)$, 
		i.e., $\Gamma(E)$ is a Banach module over $\uC(L)$ (cf.\ \cite[Chapter 2]{DuGi1983}).
		\item \label{remarksBB_hausdorff} If $E$ is a continuous Banach bundle, then its total space is Hausdorff 
		(see \cite[Proposition 16.4]{Gierz1982}).
		\item \label{remarksBB_lca} A Banach bundle with finite-dimensional fibers which is locally trivial as a vector bundle 
		(in the usual sense) is locally trivial as a Banach bundle since the required norm estimates 
		follow from compactness and upper semicontinuity of the norm (see \cite[Proposition 10.9]{Gierz1982}).
		\item \label{remarksBB_lcb} By \cite[Theorem 18.5]{Gierz1982}, a Banach bundle of constant finite dimension has a 
		Hausdorff total space if and only if it is locally trivial.
	\end{enumerate}
	
\end{remark}

We also recall the notion of subbundles.

\begin{definition}
  A \emph{subbundle} of a Banach bundle $E$ is a subset $F$ of $E$ together with the restricted 
  mapping $p|_F\colon F \to L$ such that the following conditions are satisfied.
  \begin{itemize}
    \item $F_l = F\cap E_l$ is a closed linear subspace of $E_l$ for every $l \in L$.
    \item The restricted mapping $p|_F$ is still open.
  \end{itemize}
\end{definition}

Under these conditions, $F$ together with $p|_F$ is again a 
Banach bundle (see \cite[Section 8]{Gierz1982}).  

There are plenty of examples of Banach bundles coming from differential geometry. Here we are 
interested in Banach bundles arising from surjections between compact spaces.

\begin{example}\label{Banachbundleexample}
  Let $q \colon K \to L$ be an open continuous surjection between compact spaces. Then a moment's 
  thought reveals that the compact-open topology on $\uC_q(K)$ agrees with the topology 
  generated by the base
	\begin{align*}
		V(F,U,\epsilon) \defeq \left\{f \in \uC_q(K)\mmid s(f) \in U, \left\|f - F|_{K_{s(f)}}\right\| < \epsilon \right\}
	\end{align*}
  for $F \in \uC(K)$, open $U \subset L$, and $\epsilon > 0$ (considered, e.g., in \cite[p.\ 30]{Knapp1967}). 
  Together with the canonical mapping $p \coloneqq s \colon \uC_q(K) \to L$, the space $\uC_q(K)$ becomes a 
  continuous Banach bundle over $L$. Moreover, the mapping
  \begin{align*}
    \uC(K) \to \Gamma\left(\uC_q(K)\right), \quad F \mapsto [l \mapsto F_l]
  \end{align*}
  is an isometric isomorphism of Banach modules over $\uC(L)$ by means of which 
  we identify the continuous sections of $\uC_q(K)$ with $\uC(K)$.
\end{example}

Next, we introduce the notion of continuous representations for topological groupoids 
(cf.\ Definition 3.1 of \cite{Bos2011}). Note here that if $E$ is a Banach bundle over a compact 
space $L$, then the space $\mathscr{G}(E)$ of all invertible bounded linear operators 
\begin{align*}
  T \colon E_l \to E_{\tilde{l}}
\end{align*}
for $l, \tilde{l} \in L$ is a subsemigroupoid of $\mathrm{C}_{p}^{p}(E,E)$ and itself a groupoid.

\begin{definition}\label{groupoidrepresentation}
	Let $\euG$ be a topological groupoid. A \emph{continuous representation} 
	of $\euG$ on a Banach bundle $E$ over $\euG^{(0)}$ is a homomorphism
		\begin{align*}
			T\colon \euG \to \mathscr{G}(E)
		\end{align*}
	of groupoids such that
	\begin{align*}
		\euG \times_{s,p} E \to E, \quad (\pzg,v) \mapsto T(\pzg)v
	\end{align*}
	is continuous.
	Moreover, we call a subset $F$ of $E$ \emph{$T$-invariant} if 
	$T(\pzg)(F\cap E_{s(\pzg)}) \subset F$ for every $\pzg \in \euG$.
\end{definition}

An important class of examples of continuous groupoid representations are 
the Koopman representations we already considered in \cref{def:koopmanrep}.

\begin{proposition}\label{representationextension}
	Let $(K, q, \euG, \varphi)$ be a groupoid action. Then the Koopman representation 
		\begin{align*}
			\euG \to \mathscr{G}\left(\mathrm{C}_q(K)\right), \quad \pzg \mapsto T_\pzg
		\end{align*}
	is a continuuous representation of $\euG$.
\end{proposition}
\begin{proof}
  We only check that the mapping
    \begin{align*}
      \euG \times_{s,s} \uC_q(K) \to \uC_q(K), \quad (\pzg,f) \mapsto T_\pzg f
    \end{align*}
  is continuous since the remaining assertions are obvious. Pick a net 
  $((\pzg_\alpha,f_\alpha))_{\alpha \in A}$ in $\euG \times_{s,s} \uC_q(K)$ converging to 
  $(\pzg,f) \in \euG \times_{s,s} \uC_q(K)$. We have to show that $T_{\pzg_\alpha}f_\alpha$ 
  converges to $T_\pzg f$ with respect to the compact-open topology.
  
  Let $((\pzg_\beta,f_\beta))_{\beta \in B}$ be a subnet and $(x_\beta)_{\beta \in B}$ be a 
  convergent net in $K$ with limit $x\in K$ that satisfies $q(x_\beta) = s(\pzg_\beta^{-1})$ for every 
  $\beta \in B$. Then $\lim_{\beta} \pzg_\beta^{-1}(x_{\beta}) = \pzg^{-1}(x)$. Since 
  $\lim_\beta f_\beta = f$,
  \begin{align*}
    \lim_{\beta} f_\beta\left(\pzg_\beta^{-1}\left(x_{\beta}\right)\right) = f\left(\pzg^{-1}(x)\right).
  \end{align*}
  This shows that $T_\varphi$ is continuous.
\end{proof}

We now state our first main result: a Peter--Weyl-type theorem for compact transitive groupoids. 
Here, a subset $F$ of a Banach bundle $E$ over a compact space $L$ is called \emph{fiberwise dense} 
if $F \cap E_l$ is dense in $E_l$ for every $l \in L$. The notion of a \emph{fiberwise total set} is 
defined analogously. Moreover, a groupoid $\euG$ is \emph{abelian} if all its isotropy groups 
$\euG_\pzu^\pzu$ for $\pzu \in \euG^{(0)}$ are abelian.

\begin{theorem}\label{pwtransitive}
	For a continuous representation $T$ of a compact transitive groupoid $\euG$ on a Banach bundle 
	$E$ over the unit space $\euG^{(0)}$ the following assertions hold.
	\begin{enumerate}[(i)]
		\item\label{pwtransitive_a} The union of all invariant subbundles of constant finite dimension is 
    fiberwise dense in $E$.
		\item\label{pwtransitive_b} If $\euG$ is abelian, then the union of all invariant subbundles 
    of constant dimension one is fiberwise total in $E$.
    \item\label{pwtransitive_c} If $F_1, F_2 \subset E$ are two subbundles of constant finite dimension,
    then their sum $F_1 + F_2$ is again an invariant subbundle of constant finite dimension.
  \end{enumerate}
\end{theorem}

\begin{remark}
	Notice that if $E$ has a Hausdorff total space (in particular, if $E$ has continuous norm), then the 
	subbundles in \cref{pwtransitive} are locally trivial (see \cref{remarksBB} (iv)).
\end{remark}

The proof of \cref{pwtransitive} uses the following lemma which reduces the 
problem to a single isotropy group.

\begin{lemma}\label{lem:defineinvbundle}
	Let $T$ be a continuous representation of a compact transitive groupoid on a Banach bundle $E$, 
	$\pzu \in \euG^{(0)}$, and $V \subset E_\pzu$ a closed $\euG_{\pzu}^\pzu$-invariant subspace. 
	Then setting $F_{r(\pzg)} \defeq T(\pzg)V$ for every $\pzg \in \euG_\pzu$ defines an invariant 
	subbundle $F \subset E$.
\end{lemma}
	\begin{proof}
		Note first that $F$ is well-defined. In fact, if $\pzg,\pzh \in \euG_\pzu$ with $r(\pzg) = r(\pzh)$, then 
  \begin{align*}
    T(\pzh)V = T(\pzg)T\left(\pzg^{-1}\pzh\right)V = T(\pzg)V
  \end{align*}
  since $\pzg^{-1}\pzh \in \euG_\pzu^\pzu$ and $V$ is $\euG_\pzu^\pzu$-invariant. Clearly, $F$ is 
  $\euG$-invariant and fiberwise closed.
  
  To show that $F$ is a subbundle of $E$, it suffices to check that $p|_F \colon F \to \euG^{(0)}$ 
  is open. So let $f\in F$ and $(\pzv_\alpha)_{\alpha\in A}$ be a convergent net in $\euG^{(0)}$
  with limit $\pzv\defeq p(f)\in \euG^{(0)}$. Since $(s, r)\colon\euG \to \euG^{(0)}\times\euG^{(0)}$ is 
  open by \cref{lem:transitive_open}, we may assume, after passing to a subnet, that there
  is a net $(\pzg_\alpha)_{\alpha\in A}$ with limit $\pzv$ such that $s(\pzg_\alpha) = \pzv$ and $r(\pzg_\alpha) = \pzv_\alpha$ 
  for every $\alpha \in A$. The net $(T(\pzg_\alpha)f)_{\alpha\in A}$ then is a net over 
  $(\pzv_\alpha)_{\alpha\in A}$ that converges to $f$, showing that $p|_F$ is open.
\end{proof}
	
\begin{proof}[Proof of \cref{pwtransitive}]
	For fixed $\pzu \in \euG^{(0)}$, use \cref{thm:PW_appl} to see that the union of all 
	finite-dimensional $\euG_\pzu^\pzu$-invariant subspaces is dense in $E_\pzu$. But by 
	\cref{lem:defineinvbundle} each of these invariant subspaces defines a $\euG$-invariant subbundle 
	of constant finite dimension which implies \ref{pwtransitive_a}. Likewise, a combination of \cref{thm:PW_appl}
	and \cref{lem:defineinvbundle} proves \ref{pwtransitive_b} and \ref{pwtransitive_c}.
\end{proof}

We can now apply \cref{pwtransitive} to representations of uniform enveloping groupoids $\euE_\uu(K, q, \euG)$ 
of groupoid actions $(K, q, \euG)$ on the Banach bundle $\mathrm{C}_q(K)$. However, as in Knapp's article \cite{Knapp1967}, we are primarily interested in results formulated in terms of the Banach space $\mathrm{C}(K)$ (instead of the Banach bundle $\mathrm{C}_q(K)$). Since $\mathrm{C}(K)$ can be identified with the space $\Gamma(\mathrm{C}_q(K))$ of continuous sections of $\mathrm{C}_q(K)$, the following remark explains why this can be achieved.


\begin{remark}\label{rem:bundlemodule}
  Let $p\colon E \to L$ be a Banach bundle and $\Gamma(E)$ its space of continuous sections. As noted 
  in \cref{remarksBB}, $\Gamma(E)$ is a $\uC(L)$-module and if $F\subset E$ is a subbundle, then $\Gamma(F)$
  is a closed submodule of $\Gamma(E)$. Conversely, if $\Gamma \subset \Gamma(E)$ is a closed submodule,
  one can define a corresponding subbundle $F_\Gamma \subset E$ as follows: For $l\in L$, let 
  $\operatorname{ev}_l\colon \Gamma(E) \to E_l$ denote the point evaluation in $l$ and set
  \begin{align*}
    F_\Gamma \defeq \bigcup_{l \in L} \operatorname{ev}_l(\Gamma) \subset E.
  \end{align*}
  It is shown in \cite[Theorem 8.6 and Remark 8.7]{Gierz1982} that this does indeed define a subbundle 
  of $E$ and that the assignments $F \mapsto \Gamma(F)$ and $\Gamma \mapsto F_\Gamma$ are 
  mutually inverse. In particular, there is a one-to-one correspondence between subbundles of 
  $E$. Thus, it is natural to try to rephrase properties 
  for subbundles in terms of their associated submodules. 
\end{remark}
We now translate the notions \enquote{invariant subbundle} and \enquote{locally trivial} to the language of modules.

\begin{definition}
	 Let $T$ is a representation of a groupoid $\euG$ on a Banach bundle $E$. Then subset $M \subset \Gamma(E)$ is called \emph{$T$-invariant} if $T_{\pzg} M|_{K_{s(\pzg)}} \subset M|_{K_{r(\pzg)}}$ for every $\pzg \in \euG$.
\end{definition}
It is easy to see that a subbundle $F \subset E$ is invariant if and only if the submodule $\Gamma(F) \subset \Gamma(E)$ is invariant.

The following result, which is a Banach bundle version of the classical Serre-Swan duality (see \cite{Swan1962}), charaterizes local triviality. Recall here that a module $\Gamma$ over a commutative unital ring $R$ is \emph{projective} if there is an $R$-module $\tilde{\Gamma}$ such that the module $\Gamma \oplus \tilde{\Gamma}$ is free, i.e., has a basis.
 
\begin{proposition}\label{lem:loctriv}
	For a Banach bundle $E$ over a compact space $L$ the following assertions are equivalent.
	\begin{enumerate}[(a)]
		\item\label{lem:loctriv_bundle} $E$ is locally trivial.
		\item\label{lem:loctriv_module} $\Gamma(E)$ is a finitely generated and projective $\mathrm{C}(L)$-module.
  \end{enumerate}
\end{proposition}
\begin{proof}
	The implication \enquote{\ref{lem:loctriv_bundle} $\Rightarrow$ \ref{lem:loctriv_module}} 
	follows directly from the Serre-Swan theorem. Conversely, assume that \ref{lem:loctriv_module} 
	holds. Using the Serre-Swan theorem a second time,
	we find a locally trivial vector bundle $F$ over $L$ and a (not necessarily continuous) 
	$\mathrm{C}(L)$-module isomorphism $T\colon \Gamma(F) \to \Gamma(E)$. We now 
	construct a bundle morphism $\Phi\colon F \to E$ from $T$ and show that it is continuous in order 
	to prove that $E$ and $F$ are isomorphic.
	
	Equip $F$ with any map $\|\cdot\|\colon F \to [0,\infty)$ turning $F$ into a Banach bundle 
	(these always exist, see \cite[Lemma 2]{Swan1962}). Then by \cref{remarksBB} \ref{remarksBB_lca}, $F$ is also 
	locally trivial as a Banach bundle. If $l \in L$ and $\sigma \in \Gamma(F)$ with $\sigma(l) = 0$, 
	then we find $k \in \N$, $h_j \in \mathrm{C}(L)$ with $h_j(l) = 0$ and $\tau_j \in \Gamma(F)$ for $j=1, \dots, k$ such that 
	$\sigma = \sum_{j=1}^k h_j \tau_j$ by \cite[Lemma 4]{Swan1962}. But then 
		\begin{align*}
			(T\sigma)(l) = \sum_{j=1}^k h_j(l) (T\tau_j)(l) = 0
		\end{align*} 
	for every $l \in L$. We therefore obtain a well-defined linear map $\Phi_l \colon F_l \to E_l$ by 
	setting $\Phi_l\sigma(l) \defeq (T\sigma)(l)$ for $\sigma \in \Gamma(F)$ and $l \in L$. Moreover, 
	since $F_l$ is finite-dimensional, $\Phi_l$ is bounded for every $l \in L$. We show as in the proof 
	of \cite[Theorem 1]{Swan1962} that 
	\begin{align*}
		\Phi \colon F \to E, \quad f \mapsto \Phi_{p(f)}f.
	\end{align*}
  is continuous. Pick $l \in L$. Since $F$ is locally trivial, we find a neighborhood $V \in  \mathcal{U}_L(l)$, sections 
  $\sigma_1,\dots,\sigma_n \in \Gamma(F)$ such that $\sigma_1(\tilde{l}),\dots,\sigma_n(\tilde{l})$ 
  define a base in $F_{\tilde{l}}$ for every $\tilde{l} \in V$, and continuous functions 
  $h_1,\dots,h_n\colon p^{-1}(V) \to \C$ such that 
  	\begin{align*}
  		f = \sum_{j=1}^n h_j(f) \sigma_j(p(f)) \textrm{ for every } f \in p^{-1}(V).
  	\end{align*}
  But then $\Phi(f) = \sum_{j=1}^n h_j(f) (T\sigma_j)(p(f))$ for every 
  $f \in p^{-1}(V)$. Since $T\sigma_j$ is continuous for every $j \in \{1,\dots,n\}$ we obtain that 
  $\Phi$ is continuous and hence a morphism of Banach bundles (see \cite[Definition 10.1, Proposition 10.2]{Gierz1982}). Moreover, $T\sigma = \Phi \circ \sigma$ for every $\sigma \in \Gamma(F)$, i.e., $T$ is the operator induced by the morphism $\Phi$ between the spaces of continuous sections (see \cite[Section 10]{Gierz1982}). However, by the bounded inverse theorem, a morphism of Banach bundles is an isomorphism if and only if the induced operator between the spaces of continuous sections is bijective (see \cite[Remark 10.19 (b)]{Gierz1982}). Thus, the bundles $E$ and $F$ are isomorphic and $E$ is locally trivial.
%
\end{proof}

With these translations we now formulate the main result of this section.

\begin{theorem}\label{mainthm}
	Let $(K,q,\euG, \phi)$ be a groupoid action by a topologically ergodic groupoid $\euG$.
  Then the following assertions are equivalent.
  \begin{enumerate}[(a)]
    \item\label{item:mainthm_pseudo} $(K,q,\euG, \phi)$ is pseudoisometric.
    \item\label{item:mainthm_bundle} The union of all locally trivial $T_\varphi$-invariant subbundles 
    is fiberwise dense in $\mathrm{C}_q(K)$. 
    \item\label{item:mainthm_module} The union of all finitely generated, projective, closed, 
    $T_\varphi$-invariant $\uC(\euG^{(0)})$-submodules of $\mathrm{C}(K)$ is dense in $\uC(K)$.
    \item\label{item:mainthm_module2} The finitely generated, projective, closed, 
    $T_\varphi$-invariant $\uC(\euG^{(0)})$-submodules of $\mathrm{C}(K)$ generate $\uC(K)$ as a $\mathrm{C}^*$-algebra.
  \end{enumerate}
  Moreover, if these assertions hold, then all locally trivial $T_\varphi$-invariant subbundles of 
  $\mathrm{C}_q(K)$ have constant finite dimension and the set of these subbundles is closed under
    finite sums.
\end{theorem}

We first prove some short and useful lemmas about locally trivial Banach bundles.

\begin{lemma}\label{compactnesstrivial}
	Let $E$ be a Banach bundle over a compact space $L$ which is locally trivial. If $M \subset E$ is 
	a bounded subset, i.e., $\sup_{e \in M} \|e\|_{p(e)} < \infty$, then it is precompact.
\end{lemma}
\begin{proof}
	We may assume $M$ to be closed. Now pick a net $(e_\alpha)_{\alpha \in A}$ in $M$. Passing to a 
	subnet, we may assume that $(p(e_\alpha))_{\alpha \in A}$ converges to some $l \in L$. By choosing 
	a local trivialization as in \cref{def:banachbundle}, the claim reduces to the case of a trivial 
	Banach bundle $L \times \C^n$ for which it is obvious.
\end{proof}

\begin{lemma}\label{closedloc}
	Let $E$ be a Hausdorff Banach bundle over a compact space $L$ and $F \subset E$ a locally 
	trivial subbundle. Then $F$ is closed.
\end{lemma}
\begin{proof}
	Take a net $(e_\alpha)_{\alpha \in A}$ in $F$ converging to some $e \in E$. There is an
	$\alpha_0 \in A$ such that $\sup_{\alpha \geq \alpha_0} \|e_\alpha\| < \infty$ and by \cref{compactnesstrivial} 
	we find a subnet converging to an element of $F$. Since $E$ is a Hausdorff space, we obtain that 
	$e \in F$.
\end{proof}

\begin{lemma}\label{lem:loctrivinv}
	Let $(K,q,\euG,\varphi)$ be a groupoid action. If $F \subset \mathrm{C}_q(K)$ is a 
	$T_\varphi$-invariant locally trivial subbundle, then $F$ is also invariant under the induced representation of $\euE_\uu(K,q,\euG)$.
\end{lemma}
\begin{proof}
	Let $F$ be a locally trivial $T_\varphi$-invariant subbundle
	and consider the set
		 \begin{align*}
    \euS \defeq \left\{\theta \in \euE_{\uu}(K,q,\euG)\mmid F_{r(\theta)} \circ \theta \in F_{s(\theta)}\right\}.
  \end{align*}
  It is clear that $\euS$ is a subsemigroupoid of $\uC_q^q(K,K)$ that contains $\euS(K,q,\euG)$. 
  We show that it is closed which implies $\euS = \euE_{\uu}(K,q,\euG)$. So take a net 
  $(\theta_\alpha)_{\alpha \in A}$ in $\euS$ converging to $\theta \in \euE_{\uu}(K, q, \euG)$ and 
  $e \in F_{r(\theta)} = \uC(K_{r(\theta)})$. Since $F$ is a subbundle, we then 
  find a continuous extension 
  $f \in \mathrm{C}(K)$ with $f|_{K_{r(\theta)}} = e$ and $f|_{K_{\pzu}} \in F_{\pzu}$ for all 
  $\pzu \in \euG^{(0)}$. Then 
  \begin{align*}
  	e \circ \vartheta = \lim_\alpha f|_{K_{r(\vartheta_\alpha})} \circ \vartheta_\alpha \in F
  \end{align*}
  by \cref{closedloc}.
\end{proof}

\begin{proof}[Proof of \cref{mainthm}]
	We first show that \ref{item:mainthm_pseudo} implies \ref{item:mainthm_bundle} and \ref{item:mainthm_module} as well as the additional claims. So assume that \ref{item:mainthm_pseudo} holds and recall that $\euE_\uu(K,q,\euG)$ is then a compact groupoid by \cref{thm:pseudoisochar} 
	and transitive by \cref{chartrans} since $\euG$ is topologically ergodic. Applying 
	\cref{pwtransitive} to the Koopman 
	representation of $\euE_\uu(K,q,\euG)$ on $\uC_q(K)$ yields that  
	the union of all $\euE_{\uu}(K,q,\euG)$-invariant subbundles of constant finite dimension is 
	fiberwise dense in $\uC_q(K)$ and that the set of these subbundles is closed under finite sums. 
	Since the Banach bundle $\uC_q(K)$ is continuous,  
	its total space is Hausdorff (see \cref{remarksBB} \ref{remarksBB_hausdorff}), and therefore these 
	Banach bundles are locally trivial by \cref{remarksBB} \ref{remarksBB_lcb}. Conversely, since every $T_\varphi$-invariant locally trivial subbundle is invariant with respect to the representation of $\euE_{\uu}(K,q,\euG)$ by \cref{lem:loctrivinv}, all its fibers are isomorphic and it therefore has to be of constant finite dimension. 
	
	To show that \ref{item:mainthm_module} holds, observe that for each locally trivial invariant subbundle 
	$F \subset \mathrm{C}_q(K)$, the set
    \begin{align*}
    \tilde{\Gamma}(F) \defeq 
    \left\{\sigma \in \Gamma(\uC_{q}(K)) \mmid \forall l\in L\colon \sigma(l)\in F_l \right\} 
    \subset \Gamma(\uC_{q}(K)),
  \end{align*}
  is a $\uC(\euG^{(0)})$-submodule of $\Gamma(\uC_{q}(K))$ which is isometrically isomorphic to 
  $\Gamma(F)$ as a Banach module over $\mathrm{C}(\euG^{(0)})$. In particular, $\tilde{\Gamma}(F)$ 
  is finitely generated and projective as a $\mathrm{C}(L)$-module (see \cref{remarksBB} and 
  \cref{lem:loctriv}) and closed in $\Gamma(\uC_{q}(K))$.
  Let $M$ be the union of all modules $\tilde{\Gamma}(F)$ where $F$ is a locally trivial subbundle. 
  Then $M$ is a $\uC(L)$-submodule 
  since the sum $F = F_1 + F_2$ of two locally trivial invariant subbundles of $F_1$ and $F_2$
  is again a locally trivial invariant subbundle and 
  $\tilde{\Gamma}(F_1) + \tilde{\Gamma}(F_2) \subset \tilde{\Gamma}(F)$. Moreover, by \cref{mainthm} 
  $M$ is stalkwise dense in the 
  sense of \cite[Definition 4.1]{Gierz1982} and via a Stone--Weierstra\ss{} theorem for bundles 
  (see \cite[Corollary 4.3]{Gierz1982}), this implies that $M$ is dense in $\Gamma(\uC_{q}(K))$. 
  Using the canonical isomorphy $\Gamma(\uC_{q}(K)) \cong \uC(K)$ noted in \cref{Banachbundleexample},
  we conclude that the union of all closed finitely generated and projective $T_\phi$-invariant 
  $\uC(L)$-submodules is dense in $\uC(K)$. Hence, \ref{item:mainthm_pseudo} also implies \ref{item:mainthm_module}.
	
	Clearly,  \cref{item:mainthm_module} implies \cref{item:mainthm_module2}. But also \ref{item:mainthm_bundle}  implies \cref{item:mainthm_module2}: Assume that \ref{item:mainthm_bundle} holds. By the Stone-Weierstraß theorem it suffices to show that the elements of finitely generated, projective, closed, 
    $T_\varphi$-invariant $\uC(\euG^{(0)})$-submodules separate the points of $K$. Let $x,y \in K$. If $q(x) \neq q(y)$, then we find $f \in \mathrm{C}(\euG^{(0})$ with $f(q(x)) \neq f(q(y))$ and therefore the elements of the submodule $\mathrm{C}(\euG^{(0)}) \cdot \mathbbm{1}$ separate $x$ and $y$. On the other hand, if $\pzu \coloneqq q(x)  = q(y)$, we find a locally trivial invariant subbundle $F \subseteq \mathrm{C}_q(K)$ and a section $\sigma \in \Gamma(F)$ with $\sigma(\pzu) (x) \neq \sigma(\pzu)(y)$. Since $\Gamma(F)$ defines a finitely generated, projective, closed, 
    $T_\varphi$-invariant $\uC(\euG^{(0)})$-submodule of $\mathrm{C}(K)$, this shows the claim.
    
    To finish the proof, we assume that \ref{item:mainthm_module2} holds. We show that $\euE_{\uu}(K,q,\euG)$ is compact
	to conclude that $(K,q,\euG)$ is pseudoisometric. By \cref{arzelaascoli}
	it suffices to show that $\euE_{\uu}(K,q,\euG)$ is equicontinuous. Recall from \cref{compactnessvsequicontinuity} 
	that this is equivalent to 
	$\{f \circ \theta\mid \theta \in \euE_{\uu}(K,q,\euG)\}$ being equicontinuous or, equivalently, 
	precompact in $\mathrm{C}_q(K)$ for every $f$ in a set generating $\mathrm{C}(K)$ as a C*-algebra.
	In particular, by \ref{item:mainthm_module2} we only have to show precompactness of this set for 
	$f \in \mathrm{C}(K)$ contained in a finitely generated, projective, and closed $T_\varphi$-invariant 
	$\mathrm{C}(\euG^{(0)})$-submodule of $\mathrm{C}(K)$. Given such a submodule $\Gamma \subset \mathrm{C}(K)$, 
	the subspaces $F_\pzu \defeq \Gamma|_{K_\pzu}$
	are finite-dimensional for $\pzu \in \euG^{(0)}$. As noted in \cref{rem:bundlemodule}, 
	they define a subbundle of the Banach bundle $\uC_q(K)$ and $\Gamma(F)$ 
	is isomorphic to $\Gamma$ as a Banach module over $\mathrm{C}(\euG^{(0)})$. Therefore $F$ is locally trivial 
	by \cref{lem:loctriv}. Clearly, it is $T_\varphi$-invariant. Since 
	$\{f \circ \theta\mid \theta \in \euE_{\uu}(K,q,\euG)\}$	is contained in $F$ by \cref{lem:loctrivinv}, 
	its precompactness follows using \cref{compactnesstrivial} since it is a bounded subset 
	of $\uC_q(K)$. Therefore, $\euE_{\uu}(K,q,\euG)$ is compact, proving \ref{item:mainthm_pseudo}.
\end{proof}

\section{Haar systems and relatively invariant measures}
\label{sec:RIM}

Using the (uniform) enveloping semigroup, it can be shown that any equicontinuous, minimal system $(K, G)$
has a unique invariant probability measure which is the pushforward of the Haar measure on the 
compact group $\uE(K, G)$. A relative version of this result also holds in the sense that,
given an equicontinuous extension $q\colon (K, G) \to (L, G)$ of minimal systems, there exists a unique 
\emph{relatively invariant measure} for the extension (see \cite[Corllary 3.7]{Glas1975}). We recall the definition and remind the reader of the notation introduced in the introduction.

\begin{definition}
  Let $(K,q,\euG)$ be a groupoid action. A \emph{relatively invariant measure} 
  for $(K,q,\euG)$ is a continuous mapping
  \begin{align*}
    \mu \colon \euG^{(0)} \to \uP(K), \quad \pzu \mapsto \mu_\pzu
  \end{align*}
  such that
    \begin{itemize}
      \item $q_*\mu_\pzu = \delta_\pzu$ for all $\pzu \in \euG^{(0)}$,
      \item $\phi_\pzg^*\mu_{r(\pzg)} = \mu_{s(\pzg)}$ for all $\pzg \in \euG$.
    \end{itemize}
  We call $\mu$ \emph{fully supported} if $\supp \mu_\pzu = K_\pzu$ for every $\pzu \in \euG^{(0)}$.
\end{definition}

\begin{example}\label{ex:disc_rim}
  Let $q\colon (\D, \phi) \to ([0,1], \id_{[0,1]})$ 
  be the extension from \cref{ex:disc} where $\phi$ is 
  the rotation with varying velocity on $\D$. Then the groupoid
  action $(\D, q, \euS(q))$ has a unique relatively invariant measure
  $\mu$ which is fully supported.
\end{example}

Relatively 
invariant measures have been studied systematically by Glasner in \cite{Glas1975}: They allow to lift
measures along extensions; they serve as a topological version of the conditional expectations 
ergodic theory relies on considerably; and as we discuss in \cref{section:fourier}, they are also essential
for the Fourier analysis of pseudoisometric extensions. In this Section, we show that unique 
relatively invariant measures exist for pseudoisometric extensions of topologically ergodic systems.
This extends previous results to a much larger class of nonminimal systems, including in particular
all transitive systems. As above, it is essential to consider extensions as groupoid actions 
to carry out this generalization.

Since the unique invariant measure of a compact transitive group action is given by the push forward
of the Haar measure, we try to adapt this argument to the groupoid case. We therefore consider 
Haar systems, a generalization of Haar measures to group bundles and more generally groupoids, see 
\cite[Definition 2.2]{Rena1980}.

\begin{definition}
  Let $\euG$ be a compact group bundle and for $u \in \euG^{(0)}$ let $m_u$ be the Haar measure on 
  the fiber group $\euG_u$. Then $\euG$ \emph{has a continuous Haar system} if the mapping 
    \begin{align*}
      \euG^{(0)} \to \C, \quad u \mapsto \int f \dm_u
    \end{align*}
  is continuous for each $f \in \uC(\euG)$.
\end{definition}

It is known that a compact group bundle $\euG$ has a continuous Haar system if and only if the 
mapping $p \colon \euG \to \euG^{(0)}$ is open (see \cite[Lemma 1.3]{Rena1991}). In particular, by \cref{lem:transitive_open}
the isotropy bundle of every 
compact transitive groupoid has a continuous Haar system. With this
knowledge, we can prove the first result of this section. Recall from \cref{definition:groupoidaction} 
that a groupoid action is called transitive if every orbit is the entire space.

\begin{theorem}\label{meangroupoidaction}
  Let $(K,q,\euG)$ be an action by a compact transitive groupoid $\euG$. Then the following assertions hold.
  	\begin{enumerate}[(i)]
  		\item\label{item:meangroupoidaction1} $(K,q,\euG)$ admits a relatively invariant measure.
  		\item\label{item:meangroupoidaction2} $(K,q,\euG)$ admits a unique relatively invariant measure measure if and only if the action $(K, q, \euG)$ is transitive. In this case, the measure is fully supported.
  	\end{enumerate}
\end{theorem}

The proof requires the following continuity lemma.

\begin{lemma}\label{integralconvscalar}
  Let $q \colon K \to L$ be a continuous open surjection between compact spaces and
  \begin{align*}
    \mu \colon L \to \uP(K), \quad l \mapsto \mu_l
  \end{align*}
  a continuous map with $q_*\mu_l = \delta_l$ for every $l \in L$. Moreover, 
  let $(f_\alpha)_{\alpha \in A}$ be a convergent net in $\uC_q(K)$ with limit $f \in \uC_q(K)$. Then 
  \begin{align*}
      \lim_\alpha  \int_{K_{s(f_\alpha)}} f_\alpha \, \mathrm{d}\mu_{s(f_\alpha)} = \int_{K_{s(f)}} f \, \mathrm{d}\mu_{s(f)}.
  \end{align*}
\end{lemma}
\begin{proof}
  Choose $F \in \uC(K)$ such that $F|_{K_{s(f)}} = f$. For each $\alpha \in A$ choose an $x_\alpha \in K_{s(f_\alpha)}$ such that
    \begin{align*}
      \left|f_\alpha(x_\alpha) - F(x_\alpha)\right| = \sup_{x \in K_{s(f_\alpha)}} \left|f_\alpha(x) - F(x)\right| .
    \end{align*}
  For each subnet of $(f_\alpha)_{\alpha \in A}$ we then find a subnet $(f_\beta)_{\beta \in B}$ such that $x = \lim_{\beta} x_\beta$ exists in $K$. But then
    \begin{align*}
      \lim_\beta \sup_{x \in K_{s(f_\beta)}} \left|f_\beta(x) - F(x)\right| = \lim_{\beta} \left|f_\beta(x_\beta) - F(x_\beta)\right| = 0.
    \end{align*}     
  As a consequence, 
    \begin{align*}
      \lim_{\alpha} \left|\int_{K_{s(f_\alpha)}} f_\alpha \, \mathrm{d}\mu_{s(f_\alpha)} -  \int_{K_{s(f_\alpha)}} F \, \mathrm{d}\mu_{s(f_\alpha)}\right|
      \leq  \lim_\alpha \sup_{x \in K_{s(f_\alpha)}} \left|f_\alpha(x) - F(x)\right| = 0,
    \end{align*}
  which implies the claim.
\end{proof}

\begin{proof}[Proof of \cref{meangroupoidaction}]
  As above, we denote the Haar measure on $\euG_\pzu^\pzu$ by $\um_\pzu$ for $\pzu \in \euG^{(0)}$. 
  In order to prove \ref{item:meangroupoidaction1}, it suffices to consider the case that $(K, q, \euG)$ is transitive
  (in which case $\euG$ is automatically transitive). To see this, note that for fixed $x\in K$,
  the orbit $\euG x$ is a closed, $\euG$-invariant subset and that $q$ restricted
  to $\euG x$ is again an open surjection (use \cref{lem:transitive_open}). 
  
  Now suppose $(K, q, \euG)$ is transitive. For $x\in K$, denote by 
    \begin{align*}
      \rho_x\colon \euG_{q(x)}^{q(x)}\to K_{q(x)}, \quad \pzg \mapsto \pzg x
    \end{align*}
  the induced map onto the orbit of $x$. Now pick a point $x_\pzu \in K_\pzu$ for each $\pzu\in \euG^{(0)}$ 
  and set
  \begin{align*}
    \mu_\pzu \defeq \left(\rho_{x_\pzu}\right)_*(\um_\pzu).
  \end{align*}
  It is clear from the transitivity of the group action of 
  $\euG_{\pzu}^{\pzu}$ on $K_\pzu$ that $\mu_\pzu$ does not depend on the choice of 
  $x_\pzu \in K_\pzu$ and that $\supp \mu_\pzu = K_\pzu$ for every $\pzu \in \euG^{(0)}$. 
  Moreover, $\phi_{\pzg}^*\mu_{s(\pzg)} = \mu_{r(\pzg)}$ for every $\pzg\in\euG$.
  
  Now take $f \in \uC(K)$. We show that $\lim_\alpha \mu_{\pzu_\alpha}(f) = \mu_\pzu(f)$ for every net 
  $(\pzu_\alpha)_{\alpha \in A}$ converging to some $\pzu \in L$. By passing to a subnet, we may assume that 
  there is a convergent net $(x_\alpha)_{\alpha \in A}$ in $K$ with limit $x\in K$ that satisfies
  $q(x_\alpha) = \pzu_\alpha$ for all $\alpha \in A$. Then $\rho_{x_\alpha} \to \rho_x$ with 
  respect to the compact-open topology and so $f\circ\rho_{x_\alpha} \to f\circ\rho_x$
  with respect to the compact-open topology. Therefore, \cref{integralconvscalar} 
  yields
  \begin{align*}
    \lim_{\alpha \in A} \left\langle f, \mu_{\pzu_\alpha}\right\rangle
    = \lim_{\alpha \in A} \left\langle f\circ \rho_{x_\alpha}, \um_{\pzu_\alpha}\right\rangle
    = \left\langle f \circ \rho_x, \um_\pzu\right\rangle
    = \left\langle f, \mu_\pzu\right\rangle.
  \end{align*}
  Hence, $\mu\colon \euG^{(0)}\to \uP(K)$ is continuous. This shows \ref{item:meangroupoidaction1} as well as the existence of a fully supported relatively invariant measure in case of a transitive action.
  
  It remains to show that there is a unique relatively invariant measure if and only if $(K, q, \euG)$
  is transitive. Since we have seen that any orbit of $K$ carries a 
  relatively invariant measure, the action must be transitive if there is only one relatively 
  invariant measure. Conversely, suppose $(K, q, \euG)$ is transitive and let $\mu$ the relatively invariant measure constructed above. Take any relatively invariant 
  measure $\nu \colon \euG^{(0)} \to \uP(K)$ for $(K, q, \euG)$ and let $u \in \euG^{(0)}$. Then 
  $\nu$ is invariant under the action of $\euG_\pzu^\pzu$. Since a transitive action of a compact group is equicontinuous 
  and minimal and therefore uniquely ergodic, $\nu_\pzu = \mu_\pzu$. Since $\pzu \in \euG^{(0)}$ 
  was arbitrary, $\mu$ is the unique relatively invariant measure for $(K, q, \euG)$.
\end{proof}

In order to apply \cref{meangroupoidaction} to a pseudoisometric groupoid action $(K,q,\euG)$ via the uniform 
enveloping groupoid, we have to understand when the induced action of $\euE_\uu(K, q, \euG)$ is transitive. As noted in \cref{rem:transitive}, the transitivity of the groupoid action $(K, q, \euE_\uu(K, q, \euG))$  
can be split into two transitivity properties:
\begin{itemize}
  \item \enquote{Transitivity in direction of $\euG^{(0)}$}: The induced action $(\euG^{(0)}, \id_{\euG^{(0)}}, 
  \euE_\uu(K, q, \euG)$ on the unit space $\euG^{(0)}$ is transitive.
  \item \enquote{Transitivity in direction of $q$}: The isotropy groups of $\euE_\uu(K, q, \euG)$ act 
  transitively on the fibers of $q$, i.e., $(K, q, \euE_\uu(K, q, \euG))$ is fiberwise transitive.
\end{itemize}
We have already shown that transitivity in direction of $\euG^{(0)}$ is equivalent to $\euG$ 
being topologically ergodic, so it remains to find a useful characterization for fiberwise 
transitivity. To this end, we introduce a notion of relative topological ergodicity 
for groupoid actions  and show that it yields the desired characterization.


\begin{definition}\label{def:relativelytoperg}
  A groupoid action $(K,q,\euG)$ is called \emph{relatively topologically ergodic} if 
  the canonical map $\alpha\colon \fix(K,q,\euG) \to \fix(\euG^{(0)}, \id, \euG)$ is an 
  isomorphism of maximal trivial factors.
\end{definition}

\begin{remark}
	A groupoid action $(K,q,\euG)$ is relatively topologically ergodic if and only if the restricted operator
		\begin{align*}
			T_q|_{\mathrm{fix}(T_\euG)} \colon \mathrm{fix}(T_\euG) \rightarrow \fix(T_\varphi), \quad f \mapsto f \circ q
		\end{align*}
	is bijective.
\end{remark}

Clearly, every topologically ergodic groupoid action is relatively topologically ergodic. However, there are groupoid actions which are only relatively topologically ergodic, but not topologically ergodic.
\begin{example}
	$\euG$ is any groupoid with compact unit space which is not topologically ergodic, then the action $(\euG^{(0)},\mathrm{id}_{\euG^{(0)}},\euG)$ is still topologically ergodic. More concretely, let $(K,G)$ be a topological dynamical system which is not topologically ergodic. Then the action of the action groupoid $G \ltimes K$ on $K$ is relatively topologically ergodic but not topologically ergodic.
\end{example}

%

The following result now relates ergodicity of a pseudoisometric groupoid action with fiberwise 
transitivity of the action of the induced uniform enveloping groupoid.

\begin{proposition}\label{prop:transitivefibers}
  For a pseudoisometric groupoid action $(K,q,\euG)$ the following assertions are equivalent.
  \begin{enumerate}[(a)]
    \item The action of $\euG$ on $K$ is relatively topologically ergodic.
    \item The action of 
    $\EuScript{E}_\uu(K,q,\euG)$ on $K$ is fiberwise transitive.
  \end{enumerate}   
\end{proposition}

The proof follows from the following lemma which provides a more explicit characterization 
of topological ergodicity in terms of orbits of the uniform enveloping groupoid.

\begin{lemma}\label{orbitsellisgroupoid2}
  Let $(K, q, \euG, \phi)$ be a pseudoisometric groupoid action. Then the following assertions hold.
  \begin{enumerate}[(i)]
    \item\label{orbitsellisgroupoid2:part1} The map
      \begin{align*}
        \euE_{\uu}(K,q,\euG) \ltimes K \to \euE_{\uu}(R_\phi), 
        \quad (\theta,x) \mapsto (x, \theta(x))
      \end{align*}
      is a surjective morphism of compact groupoids.
      \item\label{orbitsellisgroupoid2:part2} The orbits of $\euE_\uu(K, q, \euG)$ on $K$ are precisely the equivalence classes
      of $\euE_\uu(R_\phi)$.
      \item\label{orbitsellisgroupoid2:part3} For each $x\in K$
      \begin{align*}
        q\big(\euE_\uu(K, q, \euG)x\big) = \euE_\uu\big(\euG^{(0)}, \id_{\euG^{(0)}}, \euG\big)q(x).
      \end{align*}
      Moreover, $(K, q, \euG)$ is relatively topologically ergodic if and only if for each 
      $x\in K$
      \begin{align*}
        \euE_\uu(K, q, \euG)x = q^{-1}\big(\euE_\uu\big(\euG^{(0)}, \id_{\euG^{(0)}}, \euG\big)q(x)\big),
      \end{align*}
      i.e., if every $\EuScript{E}_\uu(K,q,\euG)$-invariant subset $A \subset K$ is $q$-saturated.
  \end{enumerate}
\end{lemma}
\begin{proof}
  For \ref{orbitsellisgroupoid2:part1}, notice that the set 
  \begin{align*}
    \euS \defeq \left\{\theta \in \euE_{\uu}(K, q, \euG) \mmid \forall x \in K_{s(\theta)}\colon (x,\theta(x)) \in \euE_{\uu}(R_\phi)\right\}
  \end{align*}
  is a closed subsemigroupoid of $\euE_{\uu}(K,q,\euG)$ that contains $\euS(K,q,\euG)$ and therefore $\euS = \euE_{\uu}(K,q,\euG)$. 
  Clearly, the mapping
  \begin{align*}
      \euE_{\uu}(K,q,\euG) \ltimes K  \to \euE_{\uu}(R_\phi), 
      \quad (\theta,x) \mapsto (x, \theta(x))
  \end{align*}
  is continuous and a morphism of groupoids. Since its image is a compact subsemigroupoid of $\euE_{\uu}(R_\phi)$ that contains 
  $R_\phi$, \ref{orbitsellisgroupoid2:part1} holds. Moreover, \ref{orbitsellisgroupoid2:part2} is a direct consequence of \ref{orbitsellisgroupoid2:part1}.
%
  
  For part \ref{orbitsellisgroupoid2:part3}, use \cref{lem:factorgroupoid} to see that the
  extension 
  \begin{align*}
    q\colon (K, q, \euG) \to (\euG^{(0)}, \id_{\euG^{(0)}}, \euG)
  \end{align*}
  extends to an extension 
  \begin{align*}
   q \colon (K, q, \euE_\uu(K, q, \euG)) \to (\euG^{(0)}, \id_{\euG^{(0)}}, \euE_\uu(\euG^{(0)}, \id_{\euG^{(0)}}, \euG)).
  \end{align*}
  Thus, $q$ maps orbits of $\euE_\uu(K, q, \euG)$ onto orbits of $\euE_\uu(\euG^{(0)}, \id_{\euG^{(0)}}, \euG)$.
  Now, consider the following commutative diagram:
   \begin{align*}
    \xymatrix{
      & (K, q, \euG) \ar[ld]_-{q} \ar[rd]^-{q_{\fix}^K} & \\
      (\euG^{(0)}, \id_{\euG^{(0)}}, \euG) \ar[rd]_-{q_{\fix}^{\euG^{(0)}}}& & \fix(K, q, \euG)\ar[ld]^-{\alpha}\\
      & \fix(\euG^{(0)}, \id_{\euG^{(0)}}, \euG)&
    }
  \end{align*}
  Suppose that $(K, q, \euG)$ is relatively topologically ergodic, i.e., that $\alpha$
  is an isomorphism. Then every $\EuScript{E}_\uu(K,q,\euG)$-invariant subset $A\subset K$ 
  is $q_{\fix}^K$-saturated by \ref{orbitsellisgroupoid2:part2} and since the above diagram
  commutes, it is also saturated with respect to $q_{\fix}^{\euG^{(0)}}\circ q$ and hence with respect to $q$.
  
  Conversely, suppose that every $\EuScript{E}_\uu(K,q,\euG)$-invariant subset $A\subset K$ 
  is $q$-saturated and take such a set $A$. Then 
  \begin{align*}
    q(A) = q(\euE_\uu(K, q, \euG) A) = \euE_\uu(\euG^{(0)}, \id_{\euG^{(0)}}, \euG)q(A).
  \end{align*}
  In other words, $q(A)$ is also saturated with respect to $q_{\fix}^{\euG^{(0)}}$.
  Hence, we conclude that $q_{\fix}^K$ and $q_{\fix}^{\euG^{(0)}}\circ q =\alpha \circ q_{\fix}^K$ have the same saturated
  sets, meaning that $\alpha$ has to be injective. Therefore, $\alpha$ is an isomorphism.
\end{proof}
\begin{proof}[Proof of \cref{prop:transitivefibers}]
  If (a) holds and $x \in K$, then $q^{-1}(q(x))  \subseteq \euE_\uu(K, q, \euG)x$	
  by \cref{orbitsellisgroupoid2} \ref{orbitsellisgroupoid2:part3} which yields 
  $q^{-1}(q(x)) = \euE_\uu(K, q, \euG)_{q(x)}^{q(x)}x$.
  
  Now assume that (b) holds. We take $x \in K$ and $y \in q^{-1}\big(\euE_\uu\big(\euG^{(0)}, \id_{\euG^{(0)}},\euG\big)x\big)$. 
  By \cref{lem:factorgroupoid} we find $\vartheta \in \euE_\uu(K, q, \euG)$ with $s(\vartheta) = q(x)$ and $r(\vartheta) = q(y)$. 
  But then we can apply (b) to find $\varrho \in \euE_\uu(K, q, \euG)$ with $s(\varrho) = r(\varrho) = q(y)$ 
  and $\vartheta(\varrho(x)) = y$. This shows $y = (\vartheta \circ \varrho)(x) \in \euE_\uu(K, q, \euG)x$ and 
  therefore (a) by \cref{orbitsellisgroupoid2} \ref{orbitsellisgroupoid2:part3}.
\end{proof}
%
%
%

\begin{corollary}\label{topergchar}
  For a pseudoisometric groupoid action $(K,q,\euG)$ the following assertions are equivalent.
    \begin{enumerate}[(a)]
      \item $(K, q, \euG)$ is topologically ergodic.
      \item $\euG$ is topologically ergodic and $(K,q,\euG)$ is relatively topologically ergodic.
      \item $(K,q,\euE_\uu(K,q,\euG))$ is transitive.
    \end{enumerate}
\end{corollary}

For the groupoid actions satisfying the equivalent conditions of \cref{topergchar}, we now prove 
the existence and uniqueness of relatively invariant measures.

\begin{theorem}\label{RIM}
  Let $(K,q,\euG)$ be a pseudoisometric and topologically ergodic groupoid action. Then there is a 
  unique relatively invariant measure for $(K,q,\euG)$. Moreover, this relatively invariant measure 
  is fully supported.
\end{theorem}
\begin{proof}
  The existence of a fully supported relatively invariant measure follows by combining 
  \cref{topergchar} and \cref{meangroupoidaction}. To establish uniqueness, we need to know 
  that any relatively invariant measure for $(K, q, \euG)$ also is a relatively invariant 
  measure for $(K, q, \euE_\uu(K, q, \euG))$ to apply \cref{meangroupoidaction} again. 
  This is done in \cref{riminv} below. 
\end{proof}

\begin{lemma}\label{riminv}
  Let $(K,q,\euG)$ be a groupoid action with relatively invariant measure $\mu$. 
  Then $\mu$ also is a relatively invariant measure for the groupoid action
  $(K, q, \euE_\uu(K, q, \euG))$.
\end{lemma}
\begin{proof}
  We need to show that $\theta_*\mu_{s(\theta)}= \mu_{r(\theta)}$ for every $\theta \in \euE_{\uu}(K,q,\euG)$.
  The set
  \begin{align*}
    \euS \defeq \left\{\theta \in \euE_{\uu}(K,q,\euG) \mmid \theta_*\mu_{s(\theta)}= \mu_{r(\theta)} \right\}
  \end{align*}
  is a subsemigroupoid $\euE_{\uu}(K,q,\euG)$ that contains $\euS(K,q,\euG)$. We only have to check that it is 
  closed. If $(\theta_\alpha)_{\alpha \in A}$ is a net in $\euS$ converging to 
  $\theta \in \euE_{\uu}(K,q,\euG)$ and $f \in \uC(K)$, then $\lim_\alpha T_{\theta_\alpha} f = T_\theta f$ 
  in $\uC_q(K)$ and therefore 
  $\lim_\alpha \langle T_{\theta_\alpha}f, \mu_{s(\theta_\alpha)}\rangle = \langle T_{\theta}f, \mu_{s(\theta)}\rangle$ 
  by \cref{integralconvscalar}. Thus,
  \begin{align*}
    \left\langle f, \theta_*\mu_{s(\theta)} \right\rangle 
    &= \lim_{\alpha} \left\langle T_{\theta}f, \mu_{s(\theta)}\right\rangle 
    = \lim_{\alpha} \left\langle f, (\theta_\alpha)_* \mu_{s(\theta_\alpha)} \right\rangle 
    = \lim_{\alpha} \left\langle f, \mu_{r(\theta_\alpha)} \right\rangle 
    = \left\langle f, \mu_{r(\theta)}\right\rangle.
  \end{align*}
  This shows that $\theta \in \euS$ and so $\euE_{\uu}(K,q,\euG) = \euS$.
\end{proof}

\section{Fourier analysis}
\label{section:fourier}

The classical Peter--Weyl theorem allows, given a compact group $G$ with its Haar measure $m$,
to decompose the Hilbert space $\mathrm{L}^2(G,m)$ into a canonical orthogonal sum of finite-dimensional 
$G$-invariant subspaces. These subspaces and the projections onto them are defined by means of 
the unitary dual $\hat{G}$ which consists of equivalence classes $[\pi]$ of irreducible unitary 
representations $\pi$ of $G$.
As we recall in \cref{transitivegrpaction} below, this Fourier-analytic result can 
easily be extended to a transitive action $(K, G)$ of a 
compact group $G$, allowing to similarly decompose the space $\uL^2(K, \mu)$ where 
$\mu$ denotes the unique invariant probability measure on $K$ obtained as the pushforward of the 
Haar measure $m$. The goal of this section is to generalize this to a Fourier
analysis result for actions of compact transitive groupoids which is of interest
on its own but will also be applied to uniform enveloping semigroupoids in \cref{sec:applications}.

To understand the situation for a transitive action $(K, G)$ of a compact group,
let $\mu$ denote the above-mentioned unique $G$-invariant probability measure on $K$. In order
to obtain prospective projection operators on $\uL^2(K, \mu)$, define for 
$f \in \mathrm{L}^2(K,\mu)$,  $[\pi] \in \hat{G}$, and $\mu$-a.e.\ $x \in K$ 
\begin{align*}
  (P_{[\pi]} f) (x) \defeq  \mathrm{dim}([\pi]) \int_G  \tr([\pi])(g) f(g^{-1}x) \,\mathrm{d}m(g).
\end{align*}
Here, $\dim([\pi])$ and $\tr([\pi])$ denote
the dimension and trace of $[\pi]$, respectively. With these definitions, one 
obtains the following easy consequence of the Peter--Weyl theorem.

\begin{theorem}\label{transitivegrpaction}
	Let $(K, G)$ be a transitive action of a compact group $G$ and $\mu$ its 
	unique invariant probability measure. Then the following assertions hold. 
	\begin{enumerate}[(i)]
			\item $P_{[\pi]} \in \mathscr{L}(\mathrm{L}^2(K,\mu))$ is an orthogonal projection for every 
			$[\pi] \in \hat{G}$.
			\item The ranges $\rg(P_{[\pi]})$ for $[\pi] \in \hat{G}$ are finite-dimensional, 
			invariant subspaces of $\mathrm{C}(K)$ and are pairwise orthogonal in $\mathrm{L}^2(K,\mu)$.  
			\item\label{item:transitivegrpaction_iii} For every $f \in \mathrm{L}^2(K,\mu)$ and $[\pi] \in \hat{G}$
				\begin{align*}
					\left\| P_{[\pi]} f \right\|_{\mathrm{C}(K)} \leq \|f\|_{\mathrm{L}^2(K,\mu)}.
				\end{align*}
			\item\label{item:transitivegrpaction_iv} Each $f \in \mathrm{C}(K)$ is contained in
				\begin{align*}
					\overline{\lin \left\{P_{[\pi]}f \mmid [\pi] \in \hat{G}\right\}}^{\|\cdot\|_{\uC(K)}} \subset \mathrm{C}(K).
				\end{align*}
			\item For each $f \in \mathrm{L}^2(K,\mu)$ we have
				\begin{align*}
					(f \mid f) = \sum_{[\pi] \in \hat{G}} \left( P_{[\pi]} f \mmid P_{[\pi]} f \right).
				\end{align*}
			\item Each $f \in  \mathrm{L}^2(K,\mu)$ can be decomposed into a series
					\begin{align*}
						f = \sum_{[\pi] \in \hat{G}} P_{[\pi]} f
					\end{align*}
				converging in $\mathrm{L}^2(K,\mu)$.
		\end{enumerate}
\end{theorem}
\begin{proof}
  A simple application of the Cauchy-Schwarz inequality and Fubini's theorem shows that 
  $P_{[\pi]} \in \mathscr{L}(\mathrm{L}^2(K,\mu))$ for every $[\pi] \in \hat{G}$. 
  Moreover, if $g \in G$ and $T_g \in \mathscr{L}(\mathrm{L}^2(K,\mu))$ is the Koopman operator 
  defined by $T_gf(x) \defeq f(g^{-1}x)$ for $x \in K$ and $f \in \mathrm{L}^2(K,\mu)$, 
  then $P_{[\pi]}T_g = T_g P_{[\pi]}$ since the trace of a representation
  is constant on conjugacy classes. In particular, $\rg(P_{[\pi]})$ is 
  $T_g$-invariant for every $g \in G$.

	Now, note that by Fourier analysis of compact groups 
	(see, e.g., \cite[Sections 5.2 and 5.3]{Foll2016}), the remaining assertions are 
	clear if $(K, G)$ is given by multiplication from the left on $K = G$. 
	In this case, we denote the projections by $Q_{[\pi]}$ for $[\pi] \in \hat{G}$.
  If $(K,G)$ is a general transitive action of $G$, we fix $x \in K$.  The orbit map
  \begin{align*}
    \varrho_x \colon G \to K, \quad g \mapsto gx
  \end{align*}
	then is a continuous surjection. It induces an isometry 
	$T_{\varrho_x}\in \mathscr{L}(\mathrm{C}(K), \mathrm{C}(G))$ which then extends to an isometric embedding 
	$T_{\varrho_x} \in \mathscr{L}(\mathrm{L}^2(K,\mu),\mathrm{L}^2(G,m))$. Since 
	$T_{\varrho_x}P_{[\pi]} = Q_{[\pi]}T_{\varrho_x}$ for every $[\pi] \in \hat{G}$, 
	the statements now readily extend to the general situation.
\end{proof}

In order to prove a version of \cref{transitivegrpaction} for 
transitive actions of compact groupoids, it is necessary to replace the 
unique invariant probability measure with the unique relatively invariant measure of 
\cref{meangroupoidaction}, and the induced Hilbert space with the space of 
continuous sections of a (continuous) Hilbert bundle.

\begin{definition}\label{hilbertbundle1}
	Let $q\colon K \to L$ be an open continuous surjection between compact spaces and 
	$\mu \colon L \to \mathrm{P}(K),\, l \mapsto \mu_l$ a weak*-continuous mapping with 
	$q_*\mu_l = \delta_l$ and $\supp \mu_l = K_l$ for every $l \in L$. We consider the 
	Banach bundle defined by 
		\begin{align*}
			\mathrm{L}^2_q(K,\mu) \defeq \bigcup_{l \in L} \mathrm{L}^2 \left(K_l,\mu_l\right)
		\end{align*}
	with the canonical mapping $p \colon \mathrm{L}^2_q(K,\mu) \to L$ and the topology defined by the sets
		\begin{align*}
			V(F,U,\epsilon) \defeq \left\{f \in \mathrm{L}^2_q(K,\mu) \mmid p(f) \in U, \|f-F|_{K_{p(f)}}\|_{\mathrm{L}^2(K_l,\mu_l)} < \epsilon\right\}
		\end{align*}
	for $F \in \mathrm{C}(K)$, $U \subset L$ open, and $\epsilon > 0$.
\end{definition}

\begin{remark}\label{sectionsHilbertbundle}
	In the situation of \cref{hilbertbundle1}, it is standard to check that $\mathrm{L}^2_q(K,\mu)$ 
	endowed with the natural norm mapping
	is a continuous Banach bundle. In fact, it is even a \emph{Hilbert bundle}, i.e., the map
		\begin{align*}
			(\cdot \mid \cdot) \colon \mathrm{L}^2_q(K,\mu) \times_{L} \mathrm{L}^2_q(K,\mu) \to \C, \quad 
			(f_1,f_2) \mapsto (f_1 \mid f_2)_{\mu_{p(f_1)}} \defeq \int f_1\overline{f_2} \dmu_{p(f_1)}
		\end{align*}
	is continuous. Its space of sections $\Gamma(\mathrm{L}_q^2(K,\mu))$ equipped with the 
	\enquote{vector-valued inner product}
		\begin{align*}
			(\cdot \mid \cdot)_{\mu} \colon \Gamma(\mathrm{L}^2(K,\mu)) \to \mathrm{C}(L), \quad 
			\sigma \mapsto [l \mapsto (\sigma(l)\mid \sigma(l))_{\mu_l}]
		\end{align*}
	is then a Hilbert C*-module over $\mathrm{C}(L)$ (see, e.g.,  \cite{DuGi1983} or \cite{Lanc1995} 
	for this concept). We indentify $\mathrm{C}(K)$ with a submodule of $\Gamma(\mathrm{L}^2_q(K,\mu))$ 
	via the injective $\mathrm{C}(L)$-module homomorphism
		\begin{align*}
			\mathrm{C}(K) \to \Gamma(\mathrm{L}^2_q(K,\mu)), \quad f \mapsto [l \to f|_{K_{l}}].
		\end{align*}
	By the Stone--Weierstra\ss{} theorem for bundles (see \cite[Corollary 4.3]{Gierz1982}), $\mathrm{C}(K)$ 
	is dense in $\Gamma(\mathrm{L}^2_q(K,\mu))$. 
\end{remark}

Given a transitive action $(K,q,\euG)$ of a compact groupoid $\euG$, we now obtain a 
Hilbert bundle in a canonical way by applying the conctruction of \cref{hilbertbundle1} 
to the unique relatively invariant measure of \cref{meangroupoidaction}. The space 
$\Gamma(\mathrm{L}^2_q(K,\mu))$ of continuous sections of this bundle will then take
the role the Hilbert space $\mathrm{L}^2(K,\mu)$ played in the case of a group action.  

Finally, in order to formulate a version of \cref{transitivegrpaction} for groupoid 
actions $(K, q, \euG)$, the last required ingredient is a generalization of the occurring projection 
operators $P_{[\pi]}$. Below, we first define them on each fiber of the 
Hilbert bundle $\uL^2(K, \mu)$ using the irreducible representations of the isotropy 
bundle $\Iso(\euG)$ of $\euG$. In order for the fiber operators to fit together to a
well-defined projection operator on $\Gamma(\uL^2(K, \mu))$ and to ensure its 
$\euG$-invariance, we need to enforce a compatibility condition on the irreducible
representations of $\Iso(\euG)$ that are employed.

\begin{definition}
  Let $\euG$ be a compact groupoid.
  \begin{enumerate}[(i)]
    \item If $\pi$ is an irreducible unitary representation of $\euG_\pzu^\pzu$ and 
    $\pzg \in \euG_\pzu$, we define $\pi^\pzg(\pzh) \defeq \pi(\pzg^{-1}\pzh\pzg)$ for 
    $\pzh \in \euG_{r(\pzg)}^{r(\pzg)}$. Moreover, set 
      \begin{align*}
        [\pi]^\pzg \defeq [\pi^\pzg] \in \hat{\euG_{r(\pzg)}^{r(\pzg)}}.
      \end{align*} 
    \item We call a map
  \begin{align*}
    \gamma \colon \euG^{(0)} \to \bigcup_{\pzu \in \euG^{(0)}} \hat{\euG_{\pzu}^\pzu}
  \end{align*}
  an \emph{invariant section} if
  \begin{itemize}
    \item $\gamma(\pzu) \in \hat{\euG_{\pzu}^\pzu}$ for every $\pzu \in \euG^{(0)}$.
    \item $\gamma(\pzu)^\pzg= \gamma(\pzg\pzu \pzg^{-1})$ for all $\pzu \in \euG^{(0)}$ and $\pzg \in\euG_\pzu$.
  \end{itemize}	
  Moreover, $\Gamma_\euG$ denotes the set of all such invariant sections.
  \end{enumerate}
\end{definition}

\begin{remark}\label{remark:invsect}
  If $\euG$ is a transitive compact groupoid and $\pzu \in \euG^{(0)}$ is fixed, 
  then every $[\pi] \in \hat{\euG_\pzu^\pzu}$ defines an invariant section via 
  $\gamma_{[\pi]}(\pzg^{-1}\pzu \pzg) \defeq [\pi^\pzg]$ for all 
  $\pzg \in \euG^\pzu$, and every invariant section is of this form. In this 
  case, we therefore obtain a bijection
    \begin{align*}
      \hat{\euG_u^u} \to \Gamma_{\euG}, \quad [\pi] \mapsto \gamma_{[\pi]},
    \end{align*}
  i.e., up to choosing a base point $\pzu \in \euG^{(0)}$, 
  the set $\Gamma_\euG$ can simply be seen as one of the unitary duals 
  of the isotropy groups of $\euG$.
\end{remark}

We can now define the projection operators defined by such invariant sections.

\begin{definition}
	Let $(K,q,\euG)$ be a transitive groupoid action of a compact groupoid $\euG$ with 
	unique relatively invariant measure $\mu$. For $\gamma \in \Gamma_\euG$ 
	the \emph{associated projection} $P_\gamma$ is defined by 
	$(P_{\gamma}\sigma)(\pzu) \defeq P_{\gamma(\pzu)}\sigma(\pzu)$ for 
	$\pzu \in \euG^{(0)}$ and $\sigma \in \Gamma(\mathrm{L}_q^2(K,\mu))$.
\end{definition}

We now obtain our Fourier analytic result for transitive actions of compact 
groupoids extending results of Knapp (cf.\ \cite[Theorem 1.2]{Knapp1967}). Here, 
two subsets $M_1,M_2 \subset \Gamma(\mathrm{L}^2(K,\mu))$ are called \emph{orthogonal} 
if $(\sigma_1 \mid \sigma_2)_\mu = 0$ for all $\sigma_1 \in M_1$, $\sigma_2 \in M_2$. 
Moreover, $\lin_{\mathrm{C}(\euG^{(0)})}$ denotes the linear hull with respect to the 
$\mathrm{C}(\euG^{(0)})$-module structure on $\mathrm{C}(K)$. Recall also
from \cref{sectionsHilbertbundle} that we 
identify $\mathrm{C}(K)$ with a dense submodule of $\Gamma(\mathrm{L}^2_q(K,\mu))$.

\begin{theorem}\label{thm:fouriermain}
  For a transitive action $(K, q, \euG, \varphi)$ of a compact groupoid $\euG$ with unique relatively 
  invariant measure $\mu$ the following assertions hold.
  \begin{enumerate}[(i)]
    \item\label{item:fouriermain_i} For every $\gamma \in \Gamma_{\euG}$, 
    $P_\gamma \in \mathscr{L}(\Gamma(\mathrm{L}^2_q(K,\mu)))$ 
    is a projection and a $\mathrm{C}(\euG^{(0)})$-module homomorphism.
    \item\label{item:fouriermain_ii} For every $\sigma \in \Gamma(\mathrm{L}^2_q(K,\mu))$ and 
    $\gamma \in \Gamma_{\euG}$
      \begin{align*}
        \|P_{\gamma}\sigma\|_{\mathrm{C}(K)} \leq 
        \|\sigma\|_{\Gamma(\mathrm{L}^2_q(K,\mu))}.
      \end{align*}
    \item\label{item:fouriermain_iii} The ranges $\rg(P_\gamma)$ for $\gamma \in \Gamma_{\euG}$ are closed, finitely 
    generated, projective, $T_\varphi$-invariant $\mathrm{C}(\euG^{(0)})$-submodules of 
    $\mathrm{C}(K)$ and are pairwise orthogonal in $\Gamma(\mathrm{L}_q^2(K,\mu))$.
    \item\label{item:fouriermain_iv}  Each $f \in \mathrm{C}(K)$ is contained in
      \begin{align*}
        \overline{\lin_{\mathrm{C}(\euG^{(0)})}\left\{P_\gamma f \mid \gamma \in \Gamma_{\euG}\right\}} \subset \mathrm{C}(K).
      \end{align*}
    \item\label{item:fouriermain_v} For each $\sigma \in \Gamma(\mathrm{L}_q^2(K,\mu))$
      \begin{align*}
        (\sigma \mid \sigma)_\mu = \sum_{\gamma \in \Gamma_{\euG}} \left( P_\gamma \sigma \mid P_\gamma \sigma \right)_\mu
      \end{align*}
      with convergence in $\mathrm{C}(\euG^{(0)})$.
    \item\label{item:fouriermain_vi} Each $\sigma \in \Gamma(\mathrm{L}_q^2(K,\mu))$ can be decomposed into a series
        \begin{align*}
          \sigma = \sum_{\gamma \in \Gamma_{\euG}} P_\gamma \sigma
        \end{align*}
      converging in $\Gamma(\mathrm{L}_q^2(K,\mu))$.
  \end{enumerate}
\end{theorem}
\begin{proof}
	We start with the proof of assertions \ref{item:fouriermain_i}, 
	\ref{item:fouriermain_ii}, and \ref{item:fouriermain_iii},
	so fix $\gamma \in \Gamma_{\euG}$. 
	We first show that $P_\gamma f \in \mathrm{C}(K)$ for every 
	$f \in \mathrm{C}(K)$. So let $f \in \mathrm{C}(K)$ and define 
	for every $x\in K$ the continuous function
  \begin{align*}
    F_x\colon \euG_{q(x)}^{q(x)} \to \C, \quad 
    \pzg \mapsto \tr(\gamma(q(x))(\pzg) f(\pzg^{-1}x).
  \end{align*}
	We claim that the map
	\begin{align*}
    F\colon K \to \uC_p^p(\Iso(\euG)), \quad x \mapsto F_x  
  \end{align*}
  is continuous which would imply the continuity of $P_\gamma f$ 
  via the integral continuity criterion from \cref{integralconvscalar}.
  To see that $F$ is indeed continuous, we use the continuity 
  characterization from \cref{charconv}. So let $(x_\alpha)_{\alpha \in A}$
  be a net in $K$ with limit $x\in K$, $(x_\beta)_{\beta\in B}$ be any subnet,
  and let $(\pzg_\beta)_{\beta \in B}$ be a net in $\Iso(\euG)$ such that 
  $p(\pzg_\beta) = q(x_\beta)$ for every $\beta \in B$. Since $\euG$ is 
  transitive, there is an $h_\beta \in \euG_{p(\pzg)}^{p(\pzg_\beta)}$ for every
  $\beta \in B$ and by the usual subnet arguments we may assume 
  that $(h_\beta)_{\beta\in B}$ converges to the unit 
  $p(\pzg) = q(x)$ of $\euG_{p(\pzg)}^{p(\pzg)}$.
  Since $\gamma$ is an invariant section, we obtain for every $\beta \in B$
  \begin{align*}
    \tr(\gamma(q(x_\beta))(\pzg_\beta)) 
    &= \tr(\gamma(q(h_\beta^{-1}x_\beta h_\beta))^{h_\beta}(\pzg_\beta))
    = \tr(\gamma(q(x))^{h_\beta}(\pzg_\beta))\\
    &= \tr(\gamma(q(x))(h_\beta^{-1}\pzg_\beta h_\beta).
  \end{align*}
  Therefore,
  \begin{align*}
    \lim_\beta F_{x_\beta}(\pzg_\beta)
    &= \lim_\beta \tr(\gamma(q(x_\beta))(\pzg_\beta)) f(\pzg_\beta^{-1} x_\beta) \\
    &= \lim_\beta \tr(\gamma(q(x))(h_\beta^{-1}\pzg_\beta h_\beta)) f(\pzg_\beta^{-1} x_\beta) \\
    &= \tr(\gamma(q(x))(\pzg)) f(\pzg^{-1}x) \\
    &= F_x(\pzg).
  \end{align*}
  Hence, $F$ is continuous and so $P_\gamma f \in \uC(K)$. Moreover,
  for $f\in\uC(K)$ it follows from 
  \cref{transitivegrpaction} \ref{item:transitivegrpaction_iii}
  that 
  \begin{align*}
   \|P_\gamma f\|_{\uC(K)} \leq \|f\|_{\Gamma(\uL^2(K, \mu))}.
  \end{align*}

  For $\sigma \in \Gamma(\mathrm{L}^2_q(K,\mu))$ we already know from \cref{transitivegrpaction} (ii) that $(P_\gamma \sigma)(\pzu) \in \mathrm{C}(K_{\pzu})$ for every $\pzu \in \euG^{(0)}$. Moreover, for $\epsilon >0$ there is $f_\epsilon \in \mathrm{C}(K)$ with 
   \begin{align*}
     \|f_\epsilon|_{K_\pzu} - \sigma(\pzu)\|_{\mathrm{L}^2(K_\pzu,\mu_\pzu)}\leq \epsilon
   \end{align*}
   for all $\pzu \in \euG^{(0)}$ since $\mathrm{C}(K)$ is dense in $\Gamma(\mathrm{L}^2_q(K,\mu))$. But then
     \begin{align*}
       \|(P_\gamma f_\epsilon)|_{K_\pzu} - (P_\gamma \sigma)(\pzu)\|_{\mathrm{C}(K_\pzu)} \leq \epsilon
     \end{align*}
   for every $\pzu \in \euG^{(0}$ and $\epsilon > 0$ by \cref{transitivegrpaction} (iii). Using this observation and the fact that---by the above---$P_\gamma f_\epsilon \in \mathrm{C}(K)$ for every $\epsilon >0$ it is easy to check that $P_\gamma \sigma \in \mathrm{C}(K)$.  Moreover, since $\uC(K)$ is dense in $\Gamma(\uL^2(K, \mu))$, we obtain
   \begin{align*}
    \|P_{\gamma}\sigma\|_{\mathrm{C}(K)} \leq 
        \|\sigma\|_{\Gamma(\mathrm{L}^2_q(K,\mu))}.
   \end{align*}
  for every $\sigma \in\Gamma(\mathrm{L}^2_q(K,\mu))$. This proves \ref{item:fouriermain_ii}. Due to the fiberwise
  definition of $P_\gamma$, one readily verifies that it is a 
  projection and a $\uC(\euG^{(0)})$-module homomorphism, proving
  \ref{item:fouriermain_i}.

	 
	We now prove that $\rg(P_\gamma)$ is $\euG$-invariant, i.e., 
	for all $\pzh \in \euG$
	that 
	\begin{align*}
    T_\pzh \left(\rg(P_\gamma)|_{K_{s(\pzh)}}\right) \subset \rg(P_\gamma)|_{K_{r(\pzh)}}.  
  \end{align*}
  To that end, let $\pzh\in\euG$ and note that 
  $\rg(P_\gamma)|_{K_{s(\pzh)}}$ is finite-dimensional, so since 
  $P_\gamma(\uC(K))$ is dense in $\rg(P_\gamma)$, 
  \begin{align*}
   \rg(P_\gamma)|_{K_{s(\pzh)}} = P_\gamma(\uC(K))|_{K_{s(\pzh)}}.
  \end{align*}
  Therefore, an element of $\rg(P_\gamma)|_{K_{s(\pzh)}}$ can 
  be written as $(P_\gamma f)|_{K_{s(\pzh)}}$ for some $f\in\uC(K)$.
  If an element of $\rg(P_\gamma)|_{K_{s(\pzh)}}$ is presented in this way,
  then for every $x\in K_{r(\pzh)}$
  \begin{align*}
    \left(T_\pzh (P_\gamma f)|_{K_{s(\pzh)}}\right)(x) 
    &= \dim(\gamma(s(\pzh))) \int_{\euG_{s(\pzh)}^{s(\pzh)}} \tr(\gamma(s(\pzh)))(\pzg) f(\pzg^{-1} \pzh^{-1} x) \dm_{s(\pzh)}(\pzg) \\
    &= \dim\left(\gamma(s(\pzh))^{\pzh^{-1}}\right) \int_{\euG_{r(\pzh)}^{r(\pzh)}} \tr(\gamma(s(\pzh)))(\pzh \pzg\pzh^{-1}) f(\pzh \pzg^{-1}  x)\dm_{r(\pzh)}(\pzg) \\
    &= \dim(\gamma(r(\pzh))) \int_{\euG_{r(\pzh)}^{r(\pzh)}} \tr(\gamma(r(\pzh)))(\pzg) (T_{\pzh^{-1}}f)(\pzg^{-1}  x) \dm_{r(\pzh)}(\pzg)\\
    &= \left(P_\gamma T_\pzh (f|_{K_{s(\pzh)}})\right)(x).
  \end{align*}
  Hence, $\rg(P_\gamma)$ is a $\euG$-invariant submodule of $\uC(K)$ 
  and it is also closed since it is the range of a projection. Now, to prove 
  the remaining claims in \ref{item:fouriermain_iii}, consider the 
	disjoint union $F$ of the vector spaces 
	$F_{\pzu} \defeq \rg(P_\gamma)|_{K_{\pzu}}$ for $\pzu \in \euG^{(0)}$. 
	Using the correspondence between submodules and subbundles noted in 
	\cref{rem:bundlemodule}, it follows that $F$ is a subbundle of the 
	Banach bundle $\mathrm{C}_q(K)$ and that
  \begin{align*}
    \rg(P_\gamma) \to \Gamma(F), \quad f \mapsto [\pzu \mapsto f|_{K_{\pzu}}]
  \end{align*}
  is an isometric isomorphism of Banach modules over $\mathrm{C}(\euG^{(0)})$. Since the fibers of 
  $F$ are finite-dimensional by \cref{transitivegrpaction} and all have the same dimension by 
  invariance of $F$ and transitivity of $\euG$, we conclude that $F$ is a locally trivial Banach 
  bundle (see \cref{remarksBB} (iv)). By \cref{lem:loctriv} $\Gamma(F)$ is 
  finitely generated and projective. Thus $\rg(P_{\gamma}) \cong \Gamma(F)$ is a 
  closed, finitely generated, projective, $T_\varphi$-invariant $\mathrm{C}(\euG^{(0)})$-submodule of $\mathrm{C}(K)$. Finally, 
  it is a consequence of \cref{transitivegrpaction} that the ranges of $P_{\gamma_1}$ and 
  $P_{\gamma_2}$ are orthogonal for distinct $\gamma_1,\gamma_2 \in \Gamma_{\euG}$. 
  This proves \ref{item:fouriermain_iii}.
	
  The approximation property \ref{item:fouriermain_iv} is clear 
  on each fiber of $\uC_q(K)$ by \cref{transitivegrpaction} \ref{item:transitivegrpaction_iv}, so we use the 
  Stone--Weierstrass theorem for bundles to achieve uniform approximation:
  Take $f \in \mathrm{C}(K)$ and consider the closed submodule 
  $\Gamma$ of $\mathrm{C}(K)$ generated by $f$ and every $P_\gamma f$ for 
  $\gamma \in \Gamma_{\euG}$. By \cref{Banachbundleexample} 
  and the correspondence from \cref{rem:bundlemodule} we obtain a 
  subbundle $F$ of $\mathrm{C}_q(K)$ by 
  $F_\pzu \defeq \Gamma|_{K_\pzu}$ for $\pzu \in \euG^{(0)}$ and 
 	\begin{align*}
 		\Gamma \to \Gamma(F), \quad f \mapsto \left[\pzu \mapsto f|_{K_\pzu}\right]
 	\end{align*}
 	is an isometric isomorphism of Banach modules over 
 	$\mathrm{C}(\euG^{(0)})$. By \cref{transitivegrpaction} (iv)  
 		\begin{align*}
 			\lin_{\mathrm{C}(\euG^{(0)})}\left\{P_\gamma f \mmid \gamma \in \Gamma_{\euG}\right\}
 		\end{align*}
 	defines a stalkwise dense subset of $\Gamma(F)$ in the 
	sense of \cite[Definition 4.1]{Gierz1982}. By the Stone--Weierstra\ss{} theorem \cite[Corollary 4.3]{Gierz1982}, 
	this set is dense in $\Gamma(F)$ and therefore 
 		\begin{align*}
 			\overline{\lin_{\mathrm{C}(\euG^{(0)})}\left\{P_\gamma f \mmid \gamma \in \Gamma_{\euG}\right\}} = \Gamma.
 		\end{align*}
	In particular,
 	\begin{align*}
 		f \in \overline{\lin_{\mathrm{C}(\euG^{(0)})}\left\{P_\gamma f \mid \gamma \in \Gamma_{\euG}\right\}}.
 	\end{align*}
   
  Part \ref{item:fouriermain_v} and \ref{item:fouriermain_vi} follow directly from 
  the corresponding parts of \cref{transitivegrpaction} and Dini's theorem. 
\end{proof}

\section{Applications to extensions of topological dynamical systems}
\label{sec:applications}

In this final section we translate our results on groupoid actions to extensions of topological dynamical systems. 
Recall from \cref{ex:extensiongrpdaction} that every open extension $q\colon (K,G) \to (L,G)$ 
can be equivalently described as a groupoid action $(K,q, G \ltimes L)$ and that the unit space of $G \ltimes L$ can be 
identified with $L$. In particular, we obtain a uniform enveloping semigroupoid $\euE_\uu(K,q,G \ltimes L)$ 
which we denote by $\euE_\uu(q)$ in the following.

We also remind the reader that the usual notions of structuredness of the extension 
(i.e., being stable, equicontinuous, pseudoisometric, or isometric) are equivalent to the 
corresponding concepts for the groupoid action. Finally, recall from the introduction that a 
topological dynamical system $(K, G)$ is \emph{topologically ergodic} if $\fix(T_\varphi)$ contains only the 
constant functions. This is the case if and only if the action groupoid $G \ltimes K$ is 
topologically ergodic and there are many examples for such systems (e.g., transitive systems) which are not minimal.

With these correspondences in mind, we restate our main results in the case of extensions of 
topological dynamical systems, starting with the characterization of equicontinuity via 
compactness from \cref{compactnessvsequicontinuity}. 

\begin{theorem}\label{ext1}
	For an open extension $q \colon (K,G,\varphi) \to (L,G,\psi)$ of topological 
	dynamical systems the following assertions are equivalent.
		\begin{enumerate}[(a)]
			\item $q$ is equicontinuous.
			\item $\{\varphi_g|_{K_l}\mid g \in G, l \in L\} \subset \mathrm{C}_q^q(K,K)$ is precompact.
			\item $\{(T_{\varphi_g}f)|_{K_l}\mid g \in G, l \in L\} \subset \mathrm{C}_q(K)$ is 
			equicontinuous for every $f\in \uC(K)$.
			\item $\{(T_{\varphi_g}f)|_{K_l}\mid g \in G, l \in L\} \subset \mathrm{C}_q(K)$ 
			is precompact for every $f\in \uC(K)$.
		\end{enumerate}
\end{theorem}

We have noted that---in contrast to extensions of minimal systems---there is a 
significant difference between equicontinuous and pseudoisometric extensions: 
There are examples of extensions of nonminimal systems which are equicontinuous
but not pseudoisometric, see \cref{ex:equinoniso}.
The following result, combining \cref{thm:pseudoisochar}, \cref{remark:metrizable} and \cref{mainthm}, indicates that 
pseudoisometric extensions are the most \enquote{natural} generalizations of 
almost periodic systems. 

\begin{theorem}\label{ext2}
	For an open extension $q \colon (K,G,\varphi) \to (L,G,\psi)$ of topological dynamical 
	systems such that $(L,G,\psi)$ is topologically ergodic, the following assertions are equivalent.
		\begin{enumerate}[(a)]
			\item $q$ is pseudoisometric.
			\item The uniform enveloping semigroupoid $\euE_\uu(q)$ is a compact groupoid.
			\item The union of all locally trivial $T_\varphi$-invariant subbundles is fiberwise 
			dense in $\mathrm{C}_q(K)$.
			\item The union of all finitely generated, projective, closed $T_\varphi$-invariant 
			$\uC(L)$-submodules of $\mathrm{C}(K)$ is dense in $\uC(K)$.
		\end{enumerate}
	If these assertions hold, then the locally trivial invariant subbundles in (c) are of constant 
	finite dimension. Moreover, if $K$ is metrizable, then (a) can be replaced by
		\begin{enumerate}[(a')]
			\item $q$ is isometric. 
		\end{enumerate}
\end{theorem}
\cref{ext2} shows that the known characterizations of almost periodic systems via the enveloping semigroup or the Koopman operator extend in a canonical way to extensions of dynamical systems. In particular, it provides a clear picture of (pseudo)isometric extensions from an operator theoretic point of view.

If we require both systems to be topologically ergodic, we even obtain the 
existence of relatively invariant measures, a result previously only known in
the minimal case (see, e.g. \cite[Section 3]{Glas1975} and Corollary 3.7 therein, or
\cite[Proposition 5.5]{Knapp1967}). 
Given an extension $q \colon (K,G) \to (L,G)$ of dynamical systems, a map 
$\mu \colon L \to \mathrm{P}(K)$ is called a \emph{relatively 
invariant measure for $q$}, if $\mu$ is weak*-continuous, $\supp(\mu_l) \subset K_l$,
and $g_*\mu_l = \mu_{gl}$ for every $g\in G$ and $l\in L$. A relatively
invariant measure is called \emph{fully supported} if $\supp(\mu_l) = K_l$
for each $l\in L$. It is immediate that a map $\mu\colon L\to\uP(K)$ 
is a relatively invariant measure for the extension $q\colon (K, G) \to (L, G)$
if and only if it is a relatively invariant measure for the groupoid action 
$(K, q, G\ltimes L)$. As a direct consequence of \cref{RIM}, we obtain the 
following existence result for relatively invariant measures.

\begin{theorem}\label{ext3}
	Every open pseudoisometric extension of topologically ergodic topological 
	dynamical systems has a unique and fully supported relatively invariant measure.
\end{theorem}

Finally, applying \cref{thm:fouriermain} to the uniform enveloping 
groupoid of an open, pseudoisometric extension yields Fourier analytic results 
for such extensions (cf.\ \cite[Theorem 1.2]{Knapp1967}).

\begin{theorem}\label{ext4}
	Let $q \colon (K,G,\varphi) \to (L,G,\psi)$ be an open pseudoisometric extension of ergodic 
	topological dynamical systems, $\mu$ its unique relatively invariant measure and 
	$\Gamma = \Gamma_{\euE_{\uu}(q)}$ the space of invariant sections into 
	the unitary dual of the isotropy bundle of $\euE_\uu(q)$.
	\begin{enumerate}[(i)]		
			\item For every $\gamma \in \Gamma$, $P_\gamma \in \mathscr{L}(\Gamma(\mathrm{L}^2_q(K,\mu)))$ is a projection and a 
			$\mathrm{C}(L)$-module homomorphism.
			\item The ranges $\rg(P_\gamma)$ for $\gamma \in \Gamma$ are closed, finitely generated, 
			projective, invariant $\mathrm{C}(L)$-submodules of $\mathrm{C}(K)$ and are pairwise orthogonal 
			in $\Gamma(\mathrm{L}_q^2(K,\mu))$.
			\item The inequality 
				\begin{align*}
					\left\|P_{\gamma}\sigma\right\|_{\mathrm{C}(K)} \leq \|\sigma\|_{\Gamma\left(\mathrm{L}^2_q(K,\mu)\right)}
				\end{align*}
				holds for every $\sigma \in \Gamma(\mathrm{L}^2_q(K,\mu))$ and $\gamma \in \Gamma$.
			\item Each $f \in \mathrm{C}(K)$ is contained in
				\begin{align*}
					\overline{\lin}_{\mathrm{C}(L)}\left\{P_\gamma f \mmid \gamma \in \Gamma \right\}.
				\end{align*}
			\item For each $\sigma \in \Gamma(\mathrm{L}_q^2(K,\mu))$ 
				\begin{align*}
					(\sigma \mid \sigma)_\mu = \sum_{\gamma \in \Gamma} (P_\gamma \sigma \mid P_\gamma \sigma)_\mu
				\end{align*}
				with convergence in $\mathrm{C}(L)$.
			\item Each $\sigma \in \Gamma(\mathrm{L}_q^2(K,\mu))$ can be decomposed into a series
					\begin{align*}
						\sigma = \sum_{\gamma \in \Gamma} P_\gamma \sigma
					\end{align*}
				that converges in $\Gamma(\mathrm{L}_q^2(K,\mu))$.
		\end{enumerate}
\end{theorem}

Thus, if both systems are ergodic, then a pseuoisometric extension can be decomposed functional analytically into \enquote{simple} parts.
This yields a precise understanding of the extension in terms of its Koopman representation.

\parindent 0pt
\parskip 0.5\baselineskip
\setlength{\footskip}{4ex}
\bibliographystyle{alpha}
\bibliography{./bib/bibliography} 
\footnotesize

\end{document}